\documentclass[journal]{IEEEtran}

\usepackage{cite}
\usepackage{fixltx2e}

\usepackage{graphicx}

\usepackage{subcaption}              

\usepackage[utf8]{inputenc}

\usepackage{amsfonts}
\usepackage{amsmath}

\usepackage{amssymb}

\usepackage{bm}

\usepackage{color,verbatim}
\usepackage{multirow}
\usepackage{accents}

\usepackage{theoremref}

\usepackage{flushend}

\usepackage{url}

\newcounter{rulecounter}
\newcommand{\resetrule}{ \setcounter{rulecounter}{0}}
\resetrule

\newsavebox{\selvestebox}
\newenvironment{colbox}[1]
  {\newcommand\colboxcolor{#1}%
   \begin{lrbox}{\selvestebox}%
   \begin{minipage}{\dimexpr\columnwidth-2\fboxsep\relax}}
  {\end{minipage}\end{lrbox}%
   \begin{center}
   \colorbox{\colboxcolor}{\usebox{\selvestebox}}
   \end{center}}

\definecolor{orange}{rgb}{1,0.8,0}
\definecolor{gray}{rgb}{.9,0.9,0.9}
\definecolor{darkgray}{rgb}{.3,0.3,0.3}
\definecolor{darkblue}{rgb}{.1,0.0,0.3}
\definecolor{lightblue}{rgb}{0.7,0.7,1}
\definecolor{lightred}{rgb}{1,0.7,.7}
\definecolor{lightgray}{rgb}{.95,0.95,0.95}
\definecolor{lightgreen}{rgb}{0.8,0.95,0.8}

\newcommand{\brackets}[1]{\left\{#1\right\}}
\newcommand{\bbm}[1]{{\bar{\bm #1}}}
\newcommand{\dbbm}[1]{{\dbar{\bm #1}}}
\newcommand{\dbar}[1]{{\bar{\bar{ #1}}}}
\newcommand{\tbm}[1]{{\tilde{\bm #1}}}

\newcommand{\ffield}{\mathbb{F}}
\newcommand{\cfield}{\mathbb{C}}
\newcommand{\rfield}{\mathbb{R}}

\newcommand{\rank}{\mathop{\rm rank}}

\newcommand{\dimf}{\mathop{\rm dim}_\ffield}
\newcommand{\dimr}{\mathop{\rm dim}_\rfield}

\newcommand{\spanv}{\mathop{\rm span}}
\newcommand{\spanvf}{\spanv_{\ffield}}
\newcommand{\spanvr}{\spanv_{\rfield}}
\newcommand{\spanvc}{\spanv_{\cfield}}
\newcommand{\vecv}{\mathop{\rm vec}}

\newcommand{\expected}[1]{\mathop{\textrm{E}}\brackets{#1} }

\newcommand{\st}{\mathop{\text{s.t.}}}

\newtheorem{myproposition}{Proposition}
\newtheorem{myremark}{Remark}
\newtheorem{mylemma}{Lemma}
\newtheorem{mytheorem}{Theorem}
\newtheorem{mydefinition}{Definition}
\newtheorem{mycorollary}{Corollary}

\newcommand{\hcolor}[1]{{{#1}}}

\newcommand{\ulen}{\hcolor{L}} 
\newcommand{\uind}{\hcolor{l}} 
\newcommand{\ulenind}{\hcolor{l}} 
\newcommand{\sulen}{{\bar \ulen}} 
\newcommand{\clen}{\hcolor{K}} 
\newcommand{\cind}{\hcolor{k}} 
\newcommand{\nblocks}{\hcolor{B}} 
\newcommand{\blockind}{\hcolor{b}}  
\newcommand{\blen}{N}  
\newcommand{\bsamps}{M} 
\newcommand{\sndiag}{\hcolor{d}}   

\newcommand{\spat}{\hcolor{\mathcal{K}}}  
\newcommand{\spatel}{{\hcolor{k}}}  
\newcommand{\bspat}{\hcolor{\mathcal{M}}}  
\newcommand{\bspatel}{{\hcolor{m}}}  

\newcommand{\dif}{\hcolor{\delta}} 

\newcommand{\basis}{\mathcal{S}}  
\newcommand{\tbasismatr}{\hcolor{\bm T}}  
\newcommand{\tbasismati}{\hcolor{\tbm T}}  
\newcommand{\cbasis}{\bar{\mathcal{S}}} 
\newcommand{\nbasis}{\hcolor{S}}  
\newcommand{\basisind}{{\hcolor{s}}} 
\newcommand{\basisst}{\mathcal{S}_{T}} 
\newcommand{\basissb}{\mathcal{S}_{B}^\sndiag} 
\newcommand{\basissc}{\mathcal{S}_{C}} 

\newcommand{\lsr}{\hcolor{\mathcal{K}}}   
\newcommand{\lsrm}{\hcolor{L}}       
\newcommand{\lsrnel}{\hcolor{K}}       
\newcommand{\lsrel}{\hcolor{k}}       
\newcommand{\plen}{\hcolor{N}}  
\newcommand{\csr}{\hcolor{\mathcal{K}}} 
\newcommand{\csrm}{\hcolor{L}}  
\newcommand{\csrel}{\hcolor{k}}       

\newcommand{\sidsulen}{\bm J_{\ulen}}  
\newcommand{\sidblen}{\bm J_{\blen}}  
\newcommand{\idzbsamps}{\bm F_{\bsamps}} 

\newcommand{\zvl}{\bm 0_{\sulen}}

\newcommand{\bx}{\bm x}
\newcommand{\by}{\bm y}

\newcommand{\bphi}{\bm \Phi}
\newcommand{\bsig}{\bm \Sigma} 
\newcommand{\bbsig}{\bbm \Sigma} 
\newcommand{\bbphi}{\bar{\bm \Phi}}

\newcommand{\phic}{\phi_\cfield}
\newcommand{\tphis}{\tilde \phi_\cfield}
\newcommand{\phir}{\phi_{|\aset}}
 
\newcommand{\vecdots}{\cdots}

\newcommand{\intset}[2]{ \{#1,\ldots,#2\} }
 \newcommand{\sspace}{\mathbb{S}^{\ulen}}
\newcommand{\acv}{\bm \sigma}
\newcommand{\rdel}{\tilde{\acv}}
\newcommand{\rhors}{\rho_\text{DS}}
\newcommand{\rhons}{\rho_\text{SS}}

\newcommand{\defmod}{{\hcolor{A}}}

\newcommand{\aelo}{{\hcolor{p}}} 
\newcommand{\aelt}{{\hcolor{q}}} 
\newcommand{\asindo}{{\hcolor{1}}} 
\newcommand{\asindt}{{\hcolor{2}}} 
\newcommand{\eij}{\bm E_{i,j}}

 \newcommand{\nset}{\mathbb{N}}
 \newcommand{\zset}{\mathbb{Z}}

\newcommand{\wset}{\mathcal{W}}

\newcommand{\bbv}{\bbm V}
\newcommand{\bV}{\bm V}
\newcommand{\bv}{\bm v}
\newcommand{\bbw}{\bbm W}
\newcommand{\dbbw}{\dbbm W}

 \newcommand{\aset}{\mathcal{A}}
\newcommand{\bset}{\mathcal{B}}

\newcommand{\bspace}{\mathbb{B}^{N,\nblocks}}
\newcommand{\tspace}{\mathbb{T}^{N\nblocks}}

\newcommand{\vindex}[1]{_{#1}}  
\newcommand{\bindex}[1]{[{#1}]}  
\newcommand{\ssindex}[1]{_{#1}}  
\newcommand{\sindex}[1]{}  

\newcommand{\change}[1]{{#1}}

\newcommand{\removenow}[1]{}  

\begin{document}


\title{ Compression Limits for Random Vectors with Linearly
  Parameterized Second-Order Statistics} \author{Daniel
  Romero,~\IEEEmembership{Student Member,~IEEE,} Roberto
  L\'opez-Valcarce,\\~\IEEEmembership{Member,~IEEE} and Geert
  Leus,~\IEEEmembership{Fellow,~IEEE.}
\thanks{Daniel Romero and Roberto L\'opez-Valcarce are with the
  Department of Signal Theory and Communications, University of Vigo,
  36310 Vigo, Spain.  Geert Leus is with the Faculty of EEMCS, Delft
  University of Technology, Mekelweg 4, 2628 CD Delft, The
  Netherlands.  (Email: \{dromero,valcarce\}@gts.uvigo.es,
  g.j.t.leus@tudelft.nl)}

\thanks{ This work was partially funded by the Spanish Government and
  the European Regional Development Fund (ERDF) under projects
  TACTICA, COMONSENS (CSD2008-00010) and COMPASS
  (TEC2013-47020-C2-1-R) and FPU grant AP2010-0149, and by the
  Galician Regional Government and ERDF under projects "Consolidation
  of Research Units'' (GRC2013/009), REdTEIC (R2014/037) and
  AtlantTIC. This work is further supported by NWO-STW under the VICI
  program (project 10382).  Parts of this work have been presented at
  the 2013 Inform. Theory Appl. Workshop, San Diego, California.  }
\thanks{ Copyright (c) 2014 IEEE. Personal use of this material is
  permitted.  However, permission to use this material for any other
  purposes must be obtained from the IEEE by sending a request to
  pubs-permissions@ieee.org.  }
}

\maketitle

\begin{abstract}
The class of complex random vectors whose  co\-va\-riance matrix is
linearly parameterized by a basis of Hermitian Toeplitz (HT) matrices
is considered, and the maximum  com\-pre\-ssion ratios that preserve
all second-order information are derived --- the statistics of the
uncompressed vector must be recoverable from a set of linearly
compressed observations. This kind of vectors arises naturally when
sampling wide-sense stationary random processes and features a number
of applications in signal and array processing.

Explicit guidelines to design optimal and nearly optimal schemes
operating both in a periodic and non-periodic fashion are
provided by considering two of the most common linear compression
schemes, which we classify as dense or sparse. It is seen that the
maximum compression ratios depend on the structure of the HT subspace
containing the covariance matrix of the uncompressed
observations. Compression patterns attaining these maximum ratios are
found for the case without structure as well as for the cases with
circulant or banded structure. Universal samplers are also proposed to
compress unknown HT subspaces.



\end{abstract}



~\\[-1.2cm]
\section{Preliminaries}
\label{sec:intro}

Consider the problem of estimating the second-order statistics of a
zero-mean random vector $\bx \in \cfield^{\ulen}$ from  a set of
$\clen$ linear observations collected in the vector $\by\in
\cfield^{\clen}$ given by
\begin{align}
\label{eq:samp}
\by = \bbphi\bx,
\end{align}
where $\bbphi\in \cfield^{\clen\times \ulen}$ is a known matrix and
several realizations of $\by$ may be available. This matrix may be
referred to as the \emph{compression matrix}, \emph{measurement
  matrix} or \emph{sampler}, where compression is achieved by setting
$\clen < \ulen$ (typically $\clen\ll\ulen$).  The covariance matrix
$\bsig = \expected{\bx \bx^H}$ contains the second-order statistics of
$\bx$ and is assumed to be a linear combination of the Hermitian
Toeplitz (HT) matrices in a given set $\basis=\{\bsig_0,\bsig_1
,\cdots, \bsig_{\nbasis-1}\}\subset \cfield^{\ulen\times \ulen}$, that
is, there exist some scalars $\alpha_\basisind$ such that
$\bsig=\sum_{\basisind} \alpha_\basisind\bsig_\basisind$.


This problem arises in inference operations over the second-order
statistics of a random vector with a Toeplitz covariance
matrix. Operating on the compressed observations $\by$ entails
multiple advantages due to their smaller dimension. In fact, many
research efforts in the last decades
have been pointed towards designing compression methods 
and reconstruction algorithms that allow
for sampling rate reductions. While most  efforts have been
focused on reconstructing $\bx$, there were also important advances
when only the second-order statistics of this vector are of
interest. This paper is concerned with problems of the second kind.


The \emph{compression ratio} $\rho = \ulen/\clen$ measures how much
$\bx$ is compressed.  \change{ The maximum compression ratio remains
  an open problem in many cases of interest; and most existing results
  rely on the usage of specific reconstruction algorithms (see
  Sec.~\ref{sec:rwc}).} This paper presents a general and unifying
framework built on abstract criteria where the maximum compression
ratio is defined and computed for most relevant settings. The proofs
involved in this theory are constructive, resulting in several methods
for designing optimal compression matrices.

~\\[-1.2cm]
\subsection{ Covariance Matching Formulation}
\label{sec:ccs}



The prior information restricts the structure of $\bsig$, thus
determining how much $\bx$ can be compressed. When no information at
all is available, $\bsig$ is simply constrained to be Hermitian
positive semidefinite and no compression is possible. However, if
$\bx$ contains samples from a wide-sense stationary process, the fact
that $\bsig$ is HT and positive semidefinite allows for a certain
degree of compression. More generally, $\bsig$ may be assumed to lie
in the intersection of the cone of positive semidefinite matrices and
the subspace spanned by a set of HT matrices (not necessarily positive
semidefinite)
$\basis=\{\bsig_0,\bsig_1,\cdots,\bsig_{\nbasis-1}\}\subset
\cfield^{\ulen\times \ulen}.$ This subspace, throughout referred to as
the \emph{covariance subspace}, captures the prior information
available and, intuitively, the smaller its dimension, the higher the
compression that can be reached.

Without any loss of generality, we consider  real  scalars:
\begin{align}
\label{eq:bsigdef}
\bsig = \sum_{\basisind=0}^{\nbasis-1} \alpha_\basisind\bsig_\basisind,~~\text{with}~~\alpha_\basisind \in \rfield,
\end{align}
and $\basis$ is assumed to be a linearly independent set of matrices:
\begin{align}
\sum_{\basisind=0}^{\nbasis-1} \alpha_\basisind \bsig_\basisind = \sum_{\basisind=0}^{\nbasis-1}
\beta_\basisind\bsig_\basisind~~\Rightarrow~~\alpha_\basisind = \beta_\basisind~\forall \basisind.
\end{align}
Thus, $\basis$ is a basis for the covariance subspace, which means
that the decomposition in \eqref{eq:bsigdef} is unique and,
consequently, knowing the $\alpha_\basisind$'s is equivalent to
knowing $\bsig$. Since the coefficients are real-valued and the
matrices HT, it is necessary that $\nbasis\leq 2 \ulen-1$ in order for
$\basis$ to be linearly independent. The second-order statistics of
$\by$, arranged in $\bbsig = \expected{\by \by^H}$, and those of
$\bx$, arranged in $\bsig$, are related by:
\begin{align}
\label{eq:decbbsig}
 \bbsig = \bbphi \bsig \bbphi^H = \sum_{\basisind=0}^{\nbasis-1} \alpha_\basisind
 \bbsig_\basisind,~~\text{where}~~\bbsig_\basisind = \bbphi \bsig_\basisind \bbphi^H.
\end{align}
 In other words, the expansion coefficients of $\bsig$ with respect to
 $\basis$ are those of $\bbsig$ with respect to $
 \cbasis=\{\bbsig_0,\bbsig_1,\cdots,\bbsig_{\nbasis-1}\}\subset
 \cfield^{\clen\times\clen}.$ Albeit Hermitian, the matrices in
 $\cbasis$ are not Toeplitz in general.  If the compression operation
 preserves all relevant information, then $\cbasis$ is linearly
 independent and knowing $\bbsig$ is equivalent to knowing the
 $\alpha_\basisind$'s, which in turn amounts to knowing
 $\bsig$. Conversely, if the compression is so strong that the linear
 independence is lost, then some second-order information about $\bx$
 cannot be recovered.


This paper unifies the treatment of a number of problems arising in
different applications (see Sec.~\ref{sec:apps}) by noting that they
can be stated as the estimation of a linearly parameterized covariance
matrix $\bbsig$ from the compressed observations $\by$, that is, they
admit a \emph{covariance matching}
formulation~\cite{ottersten1998matching,burg1982structured}. For
simplicity, a linear parameterization such as the one in
\eqref{eq:decbbsig} is assumed, but the results still apply to certain
non-linear parameterizations~\cite{ottersten1998matching} (see the
discussion around \thref{prop:aset}).

~\\[-1.2cm]
\subsection{ \change{Signal Acquisition}}
\label{sec:cs}

\change{Compression is particularly convenient in the acquisition
  stage since otherwise part of the resources would be devoted to
  acquire data that is afterwards discarded. For this reason, the
  literature contains many compressive acquisition and reconstruction
  procedures. Remarkable examples are sub-Nyquist sampling of
  multiband/multitone
  \cite{lin1998nonuniform,herley1999arbitrary,venkataramani2000multiband,mishali2010practice,tropp2010beyond}
  signals, compressed
  sensing~\cite{donoho2006compressed,candes2008introduction}, and array
  design for aperture synthesis
  imaging~\cite{hoctor1990coarray,moffet1968mra,pillai1985improved}. They
  differ as to which structure is assumed for the data and which
  information is deemed important.  }

\change{Most consider reconstructing a signal $\bx$ from linearly
  compressed observations $\by = \bbphi \bx$. Although this procedure
  is, in principle, possible when the goal is to estimate the
  second-order statistics of $\bx$, saving the intermediate step of
  reconstructing $\bx$ may entail computational advantages and greater
  compression ratios. This problem will be globally referred to as
  \emph{compressive covariance sampling}~(CCS). 
}

\change{These approaches (including CCS) share similar compression
  structures, classified here according to the nature of~$\bbphi$:}
\begin{itemize}
\item \change{ \emph{Sparse samplers} are those where $\bbphi$ is a
  sparse matrix. Commonly, $\bbphi$ is composed of $\clen$ different
  rows of the identity matrix $\bm I_\ulen$, thus performing a
  \emph{component selection} of $\bx$. If this selection is periodic,
  it is known as \emph{multi-coset sampling} (see Secs.~\ref{sec:trp}
  and \ref{sec:ucs:nus}).}
\item\change{ \emph{Dense samplers} are those where $\bbphi$ is a
  dense matrix.  Each component of $\by$ is therefore a\emph{ linear
    combination }of the components of $\bx$.  In the case of periodic
  samplers, $\bbphi$ is block diagonal where all diagonal blocks are
  replicas of a certain dense matrix (see Sec.~\ref{sec:trp}).}
\end{itemize}

\change{The nature of the acquisition architecture depends on the domain where
the signal of interest is defined:}
\begin{itemize}
\item\change{ \emph{Time-domain signals:} several alternatives have
  been proposed to replace \emph{analog-to-digital converters} (ADCs),
  which are known to be slow, expensive and power-hungry. Some
  examples include interleaved ADCs \cite{black1980interleaved},
  non-uniform sampling and its generalizations
  \cite{herley1999arbitrary,venkataramani2000multiband,wakin2012nonuniform},
  the random demodulator \cite{laska2007aic,tropp2010beyond}, the
  modulated wideband converter \cite{mishali2010practice} and the
  random modulator
  pre-integrator~\cite{yoo2012implementation,becker2011thesis}. We
  will globally refer to these devices as \emph{compressive-ADCs}
  (C-ADCs). Their operation is described by \eqref{eq:samp} when $\bx$
  contains the Nyquist samples of the signal of interest, which are
  not physically acquired but can be used as a convenient mathematical
  abstraction.}

\item\change{ \emph{Space-domain signals:} Compression is accomplished
  using \eqref{eq:samp}, where $\bx$ is a snapshot of the
  \emph{uncompressed array}. With sparse sampling (see
  e.g.~~\cite{hoctor1990coarray,moffet1968mra,pearson1990nearoptimal,wild1987difference,pillai1985improved,pillai1987augmented,abramovich1998completion1,abramovich1999completion2,pal2010nested}),
  only the antennas corresponding to the non-null columns of $\bbphi$
  need to be physically deployed to obtain $\by$, whereas in dense
  sampling~\cite{wang2009compressivearray,wang2010spacetime,venkateswaran2010beamforming},
  analog combiners are used to reduce the number of radio frequency
  chains.}
\end{itemize}

~\\[-1.2cm]
\subsection{Applications of CCS}
\label{sec:apps}
\change{We show how CCS can be applied to several problems that can be
  formulated using covariance matching models. These models need not
  be used for
  estimation~\cite{ottersten1998matching,burg1982structured}; they are
  simply used to capture the information to be preserved. }
\change{The most common covariance subspaces, defined in
  Sec.~\ref{sec:rcs}, are the \emph{Toeplitz subspace}, the
  \emph{circulant subspace} and the \emph{$\sndiag$-banded subspace}.
  Sparse and dense samplers have been considered in most applications,
  either in a periodic or non-periodic fashion.
 }
~\\[-.3cm]
\subsubsection{\change{Compressive Power Spectrum Estimation}} 
\change{The goal is to estimate $\bsig$ from $\by$ with the only
  constraint that it must be HT and positive semidefinite, which means
  that the covariance subspace is the Toeplitz subspace. One can
  employ any basis for this subspace, reconstruct $\bsig$ and apply a
  Fourier transform to find the power spectrum. More directly, one can
  consider the Fourier basis (see \eqref{eq:bffourier} below) where
  the coordinate $\alpha_\basisind$ in \eqref{eq:bsigdef} will
  represent the value of the power spectrum at frequency ${2\pi
    s}/({2\ulen-1})$. Assuming bounded autocorrelation supports
  enables $\sndiag$-banded subspaces~\cite{ariananda2012psd}, whereas
  a frequency domain formulation results in circulant 
  subspaces~\cite{yen2013subnyquist,lexa2011psd}.  }



\subsubsection{\change{Wideband Spectrum Sensing}}

\change{If $\bx = \sum_\basisind\sigma_\basisind
  \bx\ssindex{\basisind}$, where $ \bx\ssindex{\basisind}$ corresponds
  to a signal whose second-order statistics are known up to a scale,
  the parameters $\sigma_\basisind$ capturing the power of each
  component can be estimated based on the observations provided by a
  C-ADC~\cite{gonzalo2010compressed,romero2013wideband,romero2015cartography,romero2015online}. The
  covariance subspace is the span of the set of covariance matrices of
  the $ \bx\ssindex{\basisind}$'s.  }




\subsubsection{\change{Incoherent Imaging}} \change{Arbitrary distributions
of uncorrelated sources in the far field of a uniform linear array
(the uncompressed array) produce HT spatial covariance matrices.  The
angular spectrum can be obtained as the coefficients
$\alpha_\basisind$ in the expansion  \eqref{eq:bffourier} (see
\cite{hoctor1990coarray}), which correspond to the intensity impinging
from $2\ulen-1$ looking directions. Recent formulations have also
considered circulant
subspaces~\cite{krieger2013multicosetarrays,ariananda2013periodogram}.}

\subsubsection{\change{Sparse Spectrum Estimation}} \change{modal
  analysis can be used to identify the components of a \emph{sum of
    sinusoids} in noise (time-domain
  signals)~\cite{pal2011coprime,vaidyanathan2011coprime} or to
  estimate the \emph{direction of arrival} (DoA) of a number of point
  sources in the far field (space-domain
  signals)~\cite{pillai1985improved,pillai1987augmented,abramovich1998completion1,abramovich1999completion2,pal2010nested,shakeri2012sparseruler}
  using the compressed observations $\by$.  $\bsig$ is expanded as
  $\bsig=\sum_{\basisind=0}^{R-1}\alpha_\basisind\bm
  v(\phi_\basisind)\bm v^H(\phi_\basisind)$, where $R$ is the number
  of sinusoids/sources and $\bm v(\phi_\basisind)$ corresponds either
  to the sinusoid with frequency $\phi_\basisind$ or to the source at
  angle $\phi_\basisind$.  If $\bx$ is uniformly sampled, then $\bm
  v(\phi_\basisind)\bm v^H(\phi_\basisind)$ is Toeplitz. Since the
  angles $\phi_\basisind$ are unknown, the only structure present in
  $\bsig$ is that it is HT and positive
  semidefinite~\cite{pillai1985improved}. Therefore, $\bbphi$ must
  preserve the structure of \emph{any} Toeplitz matrix. An equivalent
  approach uses \emph{universal samplers} (see Sec.~\ref{sec:th}).}




~\\[-1.2cm]
\subsection{Related Work and Contributions} 
\label{sec:rwc}


\change{ Most works on reconstructing second-order statistics from
  compressed measurements deal with estimating Toeplitz covariance
  matrices using non-periodic sparse samplers, where the observation
  is that at least a pair of samples at each possible distance is
  required to estimate the statistics of the uncompressed
  signal~\cite{moffet1968mra,hoctor1990coarray,pillai1985improved,pal2010nested,pal2011coprime}.
  The optimal solution, termed \emph{restricted minimum redundancy
    array} or \emph{minimal sparse ruler}, was analyzed in
  \cite{redei1949representation,leech1956differences,wichmann1963difference,wild1987difference,pearson1990nearoptimal,linebarger1993sparsearrays}
  and shown to be optimal in direction
  finding~\cite{pillai1985improved}. Suboptimal, yet more structured,
  schemes were proposed
  in~\cite{wichmann1963difference,pearson1990nearoptimal,linebarger1993sparsearrays,pumphrey1993sparsearrays,pal2010nested,pal2012twodimensions1}.
}

\change{ Periodic sampling in $\sndiag$-banded subspaces was
  considered in~\cite{ariananda2012psd}, where the maximum $\rho$ was
  bounded using the conditions for unique reconstruction of a least
  squares algorithm. Suboptimal compression schemes were proposed
  in~\cite{ariananda2012psd}
  and~\cite{dominguezjimenez2013circular}. Non-periodic sparse
  sampling in circulant subspaces was considered in
\cite{yen2013subnyquist} and  \cite{ariananda2013periodogram}, where
  optimal and suboptimal designs are respectively found based on
  specific algorithms.\footnote{The initial statement in
    \cite{yen2013subnyquist,ariananda2013periodogram} uses periodic
    sampling, but their considerations in the frequency domain lead to
    non-periodic sampling.}
}

\change{
The sampler design criteria used in most of these works are tailored to
specific reconstruction algorithms.  Furthermore, their formulation is
not general enough to accommodate periodic samplers, dense samplers or
prior information. The contributions of this paper can be summarized as
follows:}
\change{
\begin{itemize}
\item We present a formal and general framework, irrespective of any
  algorithm, that establishes the conditions for a compression pattern
  to be \emph{admissible} and defines the maximum compression ratio
  based on abstract criteria.
\item Optimal sparse and dense samplers are found for most cases of
  interest. Novel designs include (non-)periodic sparse samplers for
  circulant and banded subspaces, periodic sparse samplers for
  Toeplitz subspaces and (non-)periodic dense samplers for Toeplitz,
  circulant and banded subspaces.
\item The notion of \emph{universal sampler} is proposed as the one
  preserving all second-order information for any HT
  covariance subspace.
\item We provide simple tools to assess admissibility in all linear
  and certain non-linear cases. Particularly, we show that the
  positive semidefinite nature of covariance matrices does not
  generally allow greater compression ratios. 
\end{itemize}
}

~\\[-1.4cm]
\subsection{Notation}
\label{sec:not}

 If a set $\aset$ is finite, then $|\aset|$ denotes its cardinality.
 If $\ffield$ is a field, then the $\ffield$-span of a set of matrices
 $\aset$ is defined as $\spanvf \aset = \{ \bm A \in \cfield^{P \times
   P}:\bm A= \sum_\basisind \alpha_\basisind \bm A_\basisind, ~\bm
 A_\basisind \in \aset,~\alpha_\basisind\in \ffield \}$. The
 $\ffield$-dimension of a set $\bset$, denoted as $\dimf \bset$, is
 the smallest $n\in \nset$ such that there exists some $\aset$ with
 $|\aset| = n$ such that $\bset \subset \spanvf \aset$. The image of a
 set $\aset$ through a function $\phi$ is denoted as $\phi(\aset)$.

Lowercase is used for scalars, bold lowercase
for vectors and bold capital for matrices. Superscript $^T$ stands for
transpose, $^H$ for conjugate transpose and $\otimes$ represents the
Kronecker product~\cite{bernstein2009}. The $(i,j)$ entry of the
$P\times Q$ matrix $\bm A$ is $a_{i,j}$, where we start with index
zero (that is, the top-left entry is $a_{0,0}$).
The vectorization of $\bm A$ is the vector $\vecv\{ \bm A\} = [ \bm
  a_0^T,\cdots,\bm a_{Q-1}^T]^T$, where $\bm a_j =
[a_{0,j},\cdots,a_{P-1,j}]^T$.  The 
$\sndiag$-th diagonal refers to the entries $(i,j)$ with
$j-i=\sndiag$, where $\sndiag$ is a negative, null or positive
integer.  $\eij$ is a matrix with all zeros except for a 1 at the position
$(i,j)$ and it is represented as $\bm e_i$ if it has a single column.

The symbol $\jmath$ denotes the imaginary unit and 
$(x)_N$  is the remainder of the integer division of $x\in
\zset$ by $N$, i.e., $(x)_N$ is the only element in the set $\{x +
\blockind N,~\blockind\in\zset \}\cap \{0,\ldots,N-1 \}$.

~\\[-1.2cm]
\subsection{Paper Structure}
The rest of the paper is structured as follows. Sec.~\ref{sec:th} sets
the theoretical background, where maximum compression ratios and
covariance samplers are defined. Sec.~\ref{sec:dcs} presents some
results to design covariance samplers, which are applied in
Secs.~\ref{sec:ucs} and~\ref{sec:nucs} to design universal and
non-universal covariance samplers, respectively. Asymptotic
compression ratios are discussed in Sec.~\ref{sec:acr}, whereas some
remarks and conclusions are respectively provided in
Secs.~\ref{sec:dis} and \ref{sec:con}.

\section{Theoretical Framework}
\label{sec:th}

The definition of the maximum compression ratio requires to first
decide which samplers we are willing to
accept. As explained in Sec.~\ref{sec:intro}, we are interested in
those samplers preserving all the second-order statistical information
of $\bx$, i.e., those samplers that allow to recover the statistics of
$\bx$ from the statistics of $\by$.  In order to formalize this
notion, let us start by associating the compression matrix
$\bbphi \in \cfield^{\clen\times \ulen}$ with a linear function that
relates the covariance matrices of $\bx$ and $\by$ and which is
defined as
   \begin{align}
     \label{eq:phidef}
     \begin{array}{ccccc}
\displaystyle   \spanvr \basis & \xrightarrow{~~\phi~~}
&\displaystyle  \spanvr \cbasis \\
  \bsig &\xrightarrow{~~~~~~} &\phi( \bsig) = \bbphi \bsig \bbphi^H
     \end{array}
   \end{align}
where, recall, $\basis$ is a linearly independent set of $\nbasis$ HT
matrices.\footnote{For mathematical convenience, $\phi$ is not only
  defined for positive semidefinite matrices.}  We next specify which
sampling matrices are \emph{admissible}:




\begin{mydefinition}
\thlabel{def:covs} \emph{ A matrix $\bbphi$ defines an
  \-$\basis$-covarian\-ce sampler\footnote{When the set $\basis$ is
    clear from the context, we will simply say that $\bbphi$ defines a
    \emph{covariance sampler}.}
  if the associated function $\phi$, defined in \eqref{eq:phidef}, is
  invertible.  }
\end{mydefinition}

  The maximum compression ratio is the largest value of $\ulen/\clen$ for
  which a covariance sampler $\bbphi\in\cfield^{\clen\times \ulen}$ can be
  found.  Above this value, it is not possible to consistently
  estimate the second-order statistics of $\bx$, even from an
  arbitrarily large number of realizations of $\by$, since the
  statistical identifiability\footnote{See
    \cite{romero2013covariancesampling} for a  discussion on the
    statistical identifiability in CCS.} of
  $\bsig$ is lost~\cite{lehmann}. For convenience, we will regard $\ulen$
  as given and attempt to minimize $\clen$.

\change{
One may argue that the requirement in \thref{def:covs} is too strong
since it suffices to require $\phi$ to be invertible only for those
matrices in $\spanvr \basis$ that are positive semidefinite.  More
generally, the prior information may constrain $\bsig$ to be in a
certain non-linear set $\aset$ such as the set of positive
semidefinite matrices, the set of covariance matrices of
auto-regressive processes with a given order, the non-linear sets
in~\cite{ottersten1998matching}, etc. In that case, we may
reformulate \thref{def:covs} to require $\phi$ to be invertible only
in $\aset \cap \spanvr \basis$. However, this is unnecessary, as shown
next:
}
\begin{mylemma}
\thlabel{prop:aset}
\emph{
Let $\phi$ be the function defined in \eqref{eq:phidef}, where $\basis$
is an independent set of $\nbasis$ HT matrices, let $\aset$ be
a set of matrices such that $\dimr[\aset \cap \spanvr
  \basis]=\nbasis$ and let $\phir$ be the
restriction of $\phi$ to $\aset \cap \spanvr \basis$, defined as:
   \begin{align}
     \begin{array}{ccccc}
\displaystyle~~~~~  \aset \cap \spanvr \basis & \xrightarrow{~~\phir~~}
&\displaystyle  \phi( \aset \cap \spanvr \basis)  \\
  \bsig &\xrightarrow{~~~~~~} &\phir( \bsig) = \phi( \bsig ).
     \end{array}
   \end{align}
Then, $\phi$ is invertible if and only if $\phir$ is
invertible.
}
\end{mylemma}
\begin{IEEEproof}
See Appendix~\ref{sec:paset}.
\end{IEEEproof}

\change{ Therefore, the non-linear information collected in $\aset$ is
  irrelevant from the linear compression perspective whenever
  $\dimr[\aset \cap \spanvr \basis]=\nbasis$. If this condition is not
  satisfied, one must choose a different basis $\basis'$ such that
  $\aset \cap \spanvr \basis' = \aset \cap \spanvr \basis$ and
  $\dimr[\aset \cap \spanvr \basis']=|\basis'|$, which is always
  possible.  This establishes the generality of \thref{def:covs} and
  enables us to work with covariance subspaces without further
  concerns.  }

\change{
If $\aset$ is the cone of positive semidefinite matrices, then
$\basis$ satisfies  $\dimr[\aset \cap \spanvr \basis]=\nbasis$ in most
cases of interest:
}
\change{
\begin{mylemma}
\thlabel{prop:psdm} \emph{ Let $\aset$ be the set of positive
  semidefinite matrices. Then $\dimr[\aset \cap \spanvr
    \basis]=\nbasis$ if at least one of the following conditions
  holds:
\begin{enumerate}
\item $\bsig\geq \bm 0$ for all $\bsig\in \basis$ 
\item $\exists\bsig\in\spanvr \basis $ such that $\bsig> \bm 0$
\end{enumerate}
}
\end{mylemma}
\begin{IEEEproof}
1) means that $\basis\subset[\aset \cap \spanvr \basis]$. Then
$\dimr[\aset \cap \spanvr \basis]\geq\dimr\basis=\nbasis$. Noting that
$\dimr[\aset \cap \spanvr \basis]\leq \nbasis$ for any $\basis$ shows
that $\dimr[\aset \cap \spanvr \basis]=\nbasis$. On the other hand, if
2) holds, we can assume without any loss of generality that
$\basis=\{\bsig_0,\ldots,\bsig_{\nbasis-1}\}$ where
$\bsig_0=\bsig>0$. If $\basis'=\{\bsig_0,
\bsig_1+\alpha\bsig_0,\ldots,\bsig_{\nbasis-1}+\alpha\bsig_0\}$, then
$\spanvr\basis = \spanvr\basis'$ for any $\alpha$. Choose
$\alpha=-\min_\basisind \lambda_\text{min}(\bsig_\basisind)/
\lambda_\text{min}(\bsig_0)$, with $ \lambda_\text{min}$ representing
the minimum eigenvalue. Then $\basis'$ satisfies 1), which concludes
the proof.
\end{IEEEproof}
}

\change{ Since at least one of the above sufficient conditions will be
  satisfied in all cases considered in this paper,
  \thref{prop:psdm} establishes that positive semidefiniteness plays
  no role in the compression of Toeplitz, circulant or banded
  subspaces. Hence, no compression improvements are possible in
  those~cases.  }



Clearly, a matrix $\bbphi$ may define a covariance sampler for
certain sets $\basis$ but not for others. If a matrix $\bbphi$ is a
covariance sampler for any choice of  $\basis$, we call it
\emph{universal}:
\begin{mydefinition} 
\emph{A sampling matrix $\bbphi\in \cfield^{\clen\times \ulen}$ defines a
  {universal covariance sampler} if it is an $\basis$-covariance sampler for any
  linearly independent set~$\basis$ of $\ulen \times \ulen$ HT matrices. }
\end{mydefinition}

Knowing $\basis$ is always beneficial since $\bbphi$ may be tailored
to obtain optimal compression ratios and estimation
performance. Universal samplers are motivated by those cases where
$\basis$, or even $\nbasis$, is unknown at the moment of designing the
compression matrix.  Note that other notions of universal samplers
have been introduced in different
contexts~\cite{feng1996blind,mishali2009reconstruction,baraniuk2008simpleproof,candes2008introduction,yen2013subnyquist}.

\subsection{Interpretation}
\label{sec:tf:int}

Due to the definition of domain and codomain in
\eqref{eq:phidef},  $\phi$ clearly represents a surjective map. Therefore,
the notion of invertibility actually means that $\phi$ must be
injective, that is, for any set of real coefficients $\alpha_\basisind$ and
$\beta_\basisind$,
\begin{align}
  \phi\left( \sum_\basisind \alpha_\basisind \bsig_\basisind\right) =   \phi\left( \sum_\basisind \beta_\basisind \bsig_\basisind\right)~~
  \Rightarrow~~
\alpha_\basisind = \beta_\basisind ~\forall \basisind.
\end{align}
This condition is, in turn, equivalent to 
\begin{align}
\label{eq:libbsig}
 \sum_\basisind \alpha_\basisind \bbsig_\basisind =    \sum_\basisind \beta_\basisind \bbsig_\basisind
  \Rightarrow
\alpha_\basisind = \beta_\basisind ~\forall \basisind,
\end{align}
which means that $\cbasis$ must be linearly independent. Thus,
determining whether a given matrix $\bbphi$ defines an
$\basis$-covariance sampler amounts to checking whether
$\cbasis=\phi(\basis)$ is linearly independent or not.  Alternatively,
\eqref{eq:libbsig} states that no two different linear combinations of
the matrices in $\cbasis$ can result in the same $\bbsig$, which means
that covariance samplers can also be defined as those samplers
preserving the identifiability of the coefficients $\alpha_\basisind$.

To the best of our knowledge, \thref{def:covs} is the first attempt to
formalize the design of samplers for CCS problems using abstract
criteria not depending on specific algorithms.  In the sequel, several
results will be established to determine whether a matrix 
defines a covariance sampler or, in some cases, even a universal
covariance sampler.

\subsection{Notable Covariance Subspaces}
\label{sec:rcs}

The results about covariance samplers  derived in this paper
will be particularized in Sec.~\ref{sec:nucs} for the most common
co\-va\-riance subspaces, which are defined next:

\subsubsection{Toeplitz Subspace}

A matrix is Toeplitz if it is constant along its
diagonals\cite{gray}. The set of all $\ulen\times \ulen$ HT matrices,
represented as $\sspace$, is a subspace of $\cfield^{\ulen\times \ulen}$
over the real scalar field,\footnote{The reason is that any linear
  combination with real coefficients of HT matrices is also HT. This
  statement is false for complex coefficients.} and it is the largest
subspace considered in this paper. The \emph{standard basis} of
$\sspace$ is defined as the set
\begin{align}
\label{eq:stdb}
\basisst=\{\bm I_\ulen\}\cup\{ \tbasismatr_1 ,\cdots, \tbasismatr_{\ulen-1}\}\cup
 \{\tbasismati_1,\cdots, \tbasismati_{\ulen-1}\},
\end{align}
where $\tbasismatr_\uind$ denotes the HT matrix with all zeros except for the entries
on the diagonals $+\uind$ and $-\uind$, which have ones, and $\tbasismati_\uind$
represents the HT matrix with all zeros except for the entries on the
diagonal $+\uind$, which have the imaginary unit $\jmath$, and those on
the diagonal $-\uind$, which have $-\jmath$.  Formally, 
\begin{align}
\tbasismatr_\uind &= \sidsulen^\uind + (\sidsulen^\uind)^T~~~~\uind\geq 1\\
\tbasismati_\uind &= \jmath\sidsulen^\uind -\jmath (\sidsulen^\uind)^T~~~~\uind\geq 1,
\end{align}
where $\sidsulen$ is the first linear shift of $\bm I_{\ulen}$ to the
right, i.e., the matrix whose element $(m,n)$ is one if $n-m = 1$ and
zero otherwise. The basis $\basisst$ shows that
$\dimr\sspace=2\ulen-1$.  Another important basis for this subspace is
the Fourier basis:
\begin{align}
\label{eq:bffourier}
\basis_F=\{\bsig_0,\cdots,\bsig_{2\ulen-2}\},~~(\bsig_\basisind)_{m,n}=
\frac{\displaystyle e^{j\frac{2\pi}{2\ulen-1}(m-n)\basisind}}{2\ulen-1}.
\end{align}

\subsubsection{Circulant Subspace} 
A circulant matrix is a matrix whose $n$-th row equals the $n$-th
circular rotation of the zeroth row\footnote{Recall the conventions
  introduced in Sec.~\ref{sec:not}.}  to the right~\cite{gray}. In
other words, the element $(m,n)$ equals the element $(m',n')$ if
$(m-n)_\ulen = (m'-n')_\ulen$.  In our case, the matrices in the
circulant subspace must be HT and circulant simultaneously. A possible
basis for $\ulen$ odd is
\begin{align}
\label{eq:bsetscdef}
\basissc=\{\bm I_\ulen\} \cup \{ \bm C_1,\cdots,\bm C_{\frac{\ulen-1}{2}}\}\cup
 \{\tbm
 C_1,\cdots,\tbm C_{\frac{\ulen-1}{2}}\},
\end{align}
where 
\begin{align}
\bm C_\ulenind &= \tbasismatr_\ulenind + \tbasismatr_{\ulen-\ulenind},~~~\ulenind=1,\ldots,\lfloor{{(\ulen-1)}/{2}}\rfloor\\
\tbm C_\ulenind &= \tbasismati_\ulenind - \tbasismati_{\ulen-\ulenind},~~~\ulenind=1,\ldots,\lfloor{{(\ulen-1)}/{2}}\rfloor,
\end{align}
and
\begin{align*}
\basissc=\{\bm I_\ulen\} \cup \{ \bm C_1,\cdots,\bm C_{\frac{\ulen}{2}-1}\}\cup
 \{\tbm
 C_1,\cdots,\tbm C_{\frac{\ulen}{2}-1}\} \cup \{ \tbasismatr_{\frac{\ulen}{2}}\}
\end{align*}
for $\ulen$ even.  Clearly, the dimension of this
 subspace equals $\ulen$.

\subsubsection{$\sndiag$-banded Subspace}
A $\sndiag$-\emph{banded matrix} is a
matrix where all the elements above the diagonal $+\sndiag$ and below
the diagonal $-\sndiag$ (these diagonals noninclusive) are zero. A
possible basis for this subspace is given by
\begin{align}
\label{eq:bsetsbdef}
\basissb=\{\bm I_\ulen\} \cup \{ \tbasismatr_1,\cdots,\tbasismatr_{\sndiag}\}\cup
 \{\tbasismati_1,\cdots,\tbasismati_{\sndiag}\},
\end{align}
which is a subset of $\basisst$. The dimension is therefore
$2\sndiag+1$.

\subsection{ The Role of Periodicity }
\label{sec:trp}

The fact that many sampling schemes operate repeatedly on a
block-by-block basis leads to the concept of periodicity (see
Sec.~\ref{sec:apps}). Note, however, that subsequent stages may
process multiple blocks jointly. Assume that $\bx$ is partitioned into
$\nblocks$ blocks of $N = \ulen/\nblocks$ samples
as\footnote{For simplicity, we assume that $\ulen$ is an integer
  multiple of $\nblocks$.} $\bx =
[\bx\bindex{0}^T,\cdots,\bx\bindex{\nblocks-1}^T]^T$, with
$\bx\bindex{\blockind}\in \cfield^{N}~\forall \blockind$ and that
sampling a block with $N$ elements results in another block with $M$
elements:
\begin{align}
\label{eq:blocksamp}
\by\bindex{\blockind}= \bphi \bx\bindex{\blockind},~~~~\blockind=0,1,\ldots,\nblocks-1,
\end{align}
where $\by\bindex{\blockind} \in \cfield^M$ and
$\bphi\in\cfield^{M\times N}$.  The use of the term \emph{periodicity}
owes to the fact that the matrix $\bphi$ does not depend on
$\blockind$.  By making $\by =
[\by\bindex{0}^T,\cdots,\by\bindex{\nblocks-1}^T]^T$ and
\begin{align}
\label{eq:ilphi}
\bbphi=\bm I_\nblocks\otimes\bphi,
\end{align}
expression \eqref{eq:blocksamp} results in \eqref{eq:samp}. From
\eqref{eq:ilphi}, it also follows that the matrices in $\cbasis$ are
block Toeplitz with $M\times M $ blocks.  

Since $\clen=M\nblocks$, the compression ratio in the periodic setting
takes the form
\begin{align}
\rho = \frac{\ulen}{\clen} = \frac{N}{M}.
\end{align}

Further conventions are useful when dealing with sparse sampling, in
which case, as seen in Sec.~\ref{sec:cs}, $\bbphi$ equals a submatrix
of $\bm I_\ulen$ up to row permutations. For concreteness, assume that
the rows of $\bbphi$ are ordered as they are in $\bm I_\ulen$.  If
$\spat=\{\uind_0,\cdots,\uind_{\clen-1}\}$ denotes the set containing
the indices of the non-null columns of $\bbphi$, the entries of $\by =
\bbphi \bx$ are given by
$y\vindex{\cind}=x\vindex{\uind_\cind},~\uind\vindex{\cind}\in \spat$,
where $\bx = [x\vindex{0},\cdots, x\vindex{\ulen-1}]^T$ and $\by =
[y\vindex{0},\cdots, y\vindex{\clen-1}]^T$.  The set $\bspat$, which
contains the indices of the non-null columns $\bphi$, is related to
$\spat$ by
\begin{align}
  \spat = \{ \bspatel+\blockind N,~\bspatel\in\bspat, \blockind=0,1,\ldots,\nblocks-1\}.
\end{align}
Loosely speaking, we say that $\spat$ is periodic with period
$\bspat$. These sets have $|\spat|=\clen=M\nblocks$ and $|\bspat|=M$
elements.

Note that periodic sampling indeed ge\-ne\-ra\-li\-zes non-periodic
sampling, since the latter can be retrieved just by making $\nblocks =
1$. For this reason, most results will be presented for periodic
samplers, with occasional comments on the non-periodic setting if
needed.

\section{Design of Covariance Samplers}
\label{sec:dcs}
The results in this section allow to determine
whether a matrix $\bbphi$ defines a covariance sampler or not,
and provide useful means to design these matrices for a given
  $\basis$. They are based on the following basic result from linear
algebra:
\begin{mylemma}
\thlabel{prop:lind} \emph{ Let $\basis=\{\bsig_0,\vecdots,\bsig_{\nbasis-1}\}$
  be a set of Hermitian matrices. If $\basis$ is
  linearly independent when considering real coefficients, that is,
\begin{align}
\label{eq:lind}
 \sum_{\basisind=0}^{\nbasis-1} \alpha_\basisind \bsig_\basisind =    \bm 0,~~~\alpha_\basisind \in \rfield~~
  \Rightarrow
\alpha_\basisind = 0 ~\forall \basisind,
\end{align}
then it is also independent when
considering coefficients in $\cfield$, i.e., \eqref{eq:lind} also applies
when $\alpha_\basisind \in \cfield$.
}
\end{mylemma}
\begin{IEEEproof}
It easily follows by combining expression \eqref{eq:lind} with the
fact that $\bsig_\basisind = \bsig_\basisind^H,~\forall \basisind$.
\end{IEEEproof}

The importance of this basic fact
is that it allows us to focus on the complex
extension of $\phi$, defined as
   \begin{align}
\label{eq:phic}
     \begin{array}{ccccc}
\displaystyle   \spanvc \basis & \xrightarrow{~~\phic~~}
&\displaystyle  \spanvc \cbasis \\
  \bsig &\xrightarrow{~~~~~~} &\phic( \bsig) = \bbphi \bsig \bbphi^H.
     \end{array}
   \end{align}
In other words,  $\bbphi$ defines a covariance sampler
iff $\phic$ is an invertible function. An equivalent statement is
provided by the following lemma, which is the basic tool to be used in
the design of  covariance samplers.

\begin{mylemma}
\thlabel{prop:ker}
\emph{
Let $\ker\phic$ denote the set of matrices $\bsig\in \spanvc \basis$
satisfying  $\phic(\bsig)=\bm 0$. Then, a matrix $\bbphi$ defines a
covariance sampler if and only if  $\ker \phic=\{\bm 0\}$.
}
\end{mylemma}
\begin{proof}
It is an immediate consequence of  \thref{def:covs} and 
\thref{prop:lind}. 
\end{proof}



\subsection{Design of Sparse Samplers}

Designing {sparse} samplers involves manipulating \emph{difference
  sets}, which contain all possible distances between elements of
another set:
\begin{mydefinition}
\emph{ The difference set of $\aset\subset \zset$, denoted as
  $\Delta(\aset)$, is defined as:}
\begin{align}
\Delta(\aset) = \{ \dif \geq 0: \exists a_1,a_2\in \aset~ \st ~\dif =a_2-a_1\}.
\end{align}
\end{mydefinition}
Note that the difference set considers no
repetition of elements, i.e., every distance shows up at most
once. The cardinality of $\Delta(\aset)$ is
upper bounded by one plus the number of unordered subsets of $\aset$
with two elements:
\begin{align}
\label{eq:cdaub}
|\Delta(\aset)|  \leq \frac{|\aset|\cdot (|\aset|-1)}{2} +1,
\end{align}
where the $+1$ term accounts for the fact that $0\in\Delta(\aset)$ for
any non-empty $\aset$. 


The correlation vector  $\acv_\basisind$ associated with the HT matrix $\bsig_\basisind$ is
defined as the  the first column of
$\bsig_\basisind$. The following theorem is a quick method to verify whether a
{sparse} sampler defined by a set $\spat$ is a covariance sampler.

\begin{mytheorem}
\thlabel{prop:rmat} \emph{ Let
  $\basis=\{\bsig_\basisind\}_{\basisind=0}^{\nbasis-1}$ be a linearly
  independent~set of HT matrices, let
  $\{\acv_\basisind\}_{\basisind=0}^{\nbasis-1}$ be the associated set
  of correlation vectors, and let $\rdel_\basisind$ be the vector
  whose entries are the elements of $\acv_\basisind$ indexed by
  $\Delta(\spat)$. Then, $\spat $ defines an $\basis$-covariance
  sampler if and only if $\rank \bm R = \nbasis$, where }
\begin{align}
\label{eq:rmat}
\bm R = 
\left[
\begin{array}{cccc}
\rdel_0 & \rdel_1 & \ldots &\rdel_{\nbasis-1}\\
\rdel_0^\ast & \rdel_1^\ast & \ldots &\rdel_{\nbasis-1}^\ast\\
\end{array}
\right].
\end{align}
\end{mytheorem}

\begin{proof}
Observe that 
 $\bbsig_\basisind$ contains an element from the $\dif$-th diagonal of
$\bsig_\basisind$ iff $|\dif |\in \Delta(\spat)$. Now vectorize the
matrices in $\cbasis$ and arrange these vectors as columns of a
matrix. By removing repeated rows and duplicating the row
corresponding to the main diagonal we obtain $\bm R$. Therefore, the
number of linearly independent columns in $\bm R$ equals the number of
linearly independent matrices in $\cbasis$.  The result follows from
\thref{prop:ker} by noting that $\ker \phic=\{\bm 0\}$ iff $\rank \bm
R = \nbasis$.
\end{proof}

From \eqref{eq:rmat}, it is easy to conclude\footnote{Note the
  existence of a duplicate row in $\bm R$.} that
$2|\Delta(\spat)|-1\geq \nbasis$ in order for $\bm R$ to be full
column rank. Combining this expression with~\eqref{eq:cdaub} results
in the following necessary condition for $\spat$ to define a
covariance sampler:
\begin{align}
  {\clen\cdot (\clen-1)} +1 \geq \nbasis.
\end{align}
\subsection{Design of {Dense} Samplers}
\label{sec:csrs}
Designing sampling matrices is oftentimes involved due to the nature
of the design criteria.  In many cases, it is convenient to draw
$\bbphi$ at random using a distribution that provides an admissible
sampler with a certain
probability~\cite{candes2008introduction,baraniuk2008simpleproof}. Following
this idea, this paper employs probabilistic techniques to obtain
optimal designs for dense samplers.

These techniques provide sampling matrices with an acceptable behavior
without considering any structure of the covariance subspace other
than its dimension.  The next result establishes the minimum size of a
random matrix $\bphi$ to define a covariance sampler.  The only
requirement is that this matrix be drawn from a continuous probability
distribution.
\change{
\begin{mytheorem}
\thlabel{prop:2} \emph{ 
Let $\bphi\in \cfield^{M\times N}$, with $M\leq N$, be a random matrix
with a continuous probability distribution.\footnote{Formally, we say
  that a distribution $\mu$ is continuous if it is \emph{absolutely
    continuous with respect to Lebesgue measure}, that is, $\mu(B)=0$
  for all Borel sets of $ \cfield^{M\times N}$ with zero Lebesgue
  measure~\cite{billingsley}. Intuitively, this means that there are
  no probability \emph{masses}. } Then, with probability one, the
matrix $\bbphi=\bm I_{\nblocks}\otimes \bphi$ defines an
$\basis$-covariance sampler if and only if $\nbasis\leq
M^2(2\nblocks-1)$, where $\nbasis$ is the cardinality of the HT basis
set $\basis$.  }
\end{mytheorem}
}
\begin{proof}
See Appendix~\ref{sec:pp2}.
\end{proof}

\removenow{Note that the requirement imposed by \thref{prop:2} on the probability
distribution is weaker than those in CS for continuous
distributions, where reconstruction results exist just for
distributions satisfying certain measure concentration
inequalities~\cite{baraniuk2008simpleproof}. Moreover, \thref{prop:2}
does not even require the elements of $\bphi$ to be i.i.d., which is a
standard requirement in CS. On the other hand, \thref{prop:2} does not
apply to the case of discrete distributions such as the Bernoulli
distribution, which is widely used in CS.}

Note that the matrices in $\spanvr\cbasis$ are Hermitian and block
Toeplitz with $M\times M$ blocks. It can be seen that the dimension of
such a subspace is at most $M^2(2\nblocks-1)$, which is exactly the
one achieved by the random design from \thref{prop:2} when $\nbasis=
M^2(2\nblocks-1)$ (see Sec.~\ref{sec:tf:int}). Therefore, no other
design can achieve a higher compression ratio.


\section{Universal Covariance Samplers}
\label{sec:ucs}

After having laid the mathematical framework, we are ready to provide
designs that result in covariance samplers independently of which
basis of HT matrices is considered.
The first result of this section reduces the task of checking whether
a given matrix defines a covariance sampler for all possible bases to
that of checking just for one. 

\begin{mylemma}
\thlabel{prop:bbasis}
\emph{
Let $\basis$ be a basis for $\sspace$. Then, a sampler $\bbphi$ is
universal if and only if it is an $\basis$-covariance sampler.
}
\end{mylemma}

\begin{proof}
Clearly, if $\bbphi$ is universal, it is also an $\basis$-covariance
sampler. Conversely, if $\bbphi$ is an $\basis$-covariance sampler, it
is also an $\basis'$-covariance sampler for any basis $\basis'$ of HT
matrices since the restriction of an injective map is always
injective.
\end{proof}

The rest of this section applies this result to obtain sparse and
dense universal covariance samplers. \removenow{The main result in the
  former case is that the period $\bspat$ of $\spat$ must be a
  \emph{sparse ruler}, which is a well-known mathematical object
  reviewed below. In the latter case, the conclusions are similar to
  those from Sec.~\ref{sec:dcs} in the sense that the universality of
  a sampler is guaranteed with probability one if the dimensions of
  $\bbphi$ are properly set.}

\subsection{Sparse Samplers}
\label{sec:ucs:nus}

The next necessary and sufficient condition for a {sparse} sampler to
be universal basically states that all autocorrelation lags must be
identifiable from the compressed observations.
\begin{mytheorem}
\thlabel{prop:1}
\emph{
The set $\spat\subset\intset{0}{\ulen-1}$ defines a universal covariance sampler if and only if
$\Delta(\spat) = \{0,\ldots,\ulen-1\}$.
}
\end{mytheorem}

\begin{proof}
\change{
Consider the basis $\basisst$ from \eqref{eq:stdb}. If
$\Delta(\spat)=\{0,\ldots,\ulen-1\}$, the matrix $\bm R$ from
\thref{prop:rmat} becomes
\begin{align}
\bm R = \left[
\begin{array}{c c}
  \bm I_\ulen &   -\jmath\tbm I_\ulen \\
  \bm I_\ulen &   \jmath\tbm I_\ulen \\
\end{array}
\right],
\end{align}
where $\tbm I_\ulen$ is the submatrix of $\bm I_\ulen$ that results
from removing the first column. Since $\bm R$ has rank $2\ulen-1$,
$\spat$ defines an $\basisst$-covariance sampler and, due to
\thref{prop:bbasis}, it is universal.
}

\change{
If one or more elements of $\{0,\ldots,\ulen-1\}$ are missing in
$\Delta(\spat)$, at least two of the rows of $\bm R$ are missing,
meaning that $\rank \bm R < 2\ulen-1$.  Then, $\rank \bm R=2\ulen-1$
iff $\{0,\ldots,\ulen-1\}\subset \Delta(\spat)$. From
\thref{prop:rmat}, $\spat$ defines an $\basisst$-covariance sampler
iff $\Delta(\spat)=\{0,\ldots,\ulen-1\}$. Now apply
\thref{prop:bbasis}.
}
\end{proof}

This theorem provides a very simple means to check whether $\spat$ is
universal or not. Interestingly, this is closely related to the
classical problem in number theory known as the \emph{sparse ruler
  problem}, or as the representation of integers by difference bases
(see~\cite{miller1971three,leech1956differences} and references
therein). Its application to array processing dates back to the
60's~\cite{moffet1968mra}. 

\begin{mydefinition}
\label{def:sr}
\emph{ A length-$(\lsrm -1)$ (linear) {sparse ruler} is a set $\lsr\subset
  \{0,1,\ldots,\lsrm-1\}$ satisfying $\Delta(\lsr)=\{0,1,\ldots,\lsrm-1\}$. It is
  called {minimal} if there exists no other length-$(\lsrm-1)$ sparse ruler
  with smaller cardinality. }
\end{mydefinition}

Intuitively, we may associate this set with a classical ruler (the
physical object) with some marks erased, which is still capable of
measuring all integer distances between 0 and its length using pairs
of marks. Two examples of minimal sparse rulers are shown in
Fig. \ref{fig:lsr}, where red dots correspond to the marks that have
not been erased.  Sparse rulers exist for all $\lsrm$, although they
are not necessarily unique. For instance, two different length-$10$
sparse rulers are $\{0, 1, 2, 3, 6, 10\}$ and $\{0,1,2,5,7,10\}$. The
most remarkable properties of a length-$(\lsrm-1)$ sparse ruler are
that the endpoints are always present, i.e., $\{0,\lsrm-1\}\subset
\lsr$, and that its reflection $(\lsrm-1)-\lsr = \{
(\lsrm-1)-\lsrel:~\lsrel\in \lsr\}$ is also a sparse ruler. Trivially,
if $\lsr$ is minimal, then $(\lsrm-1)-\lsr$ is also
minimal. Therefore, (minimal) sparse rulers exist at least in pairs
unless $\lsr = (\lsrm-1)-\lsr$. The cardinality $\lsrnel=|\lsr|$ of a
sparse ruler is lower bounded as
\begin{align}
\lsrnel\geq \frac{1}{2}+\sqrt{2(\lsrm-1)+\frac{1}{4}},
\end{align}
which follows directly from \eqref{eq:cdaub} and is only attained for
$\lsrm-1=0,1,3$ and 6 (see e.g.~\cite{linebarger1993sparsearrays}); or as
(see~\cite{redei1949representation,leech1956differences}):
\begin{align}
\label{eq:lsrub}
\lsrnel\geq  \sqrt{\tau  (\lsrm -1) },
\end{align}
where $\tau=\max_\theta{ 2(1-\frac{\sin\theta}{\theta})}\approx
2.4345$; and, if it is minimal it is upper bounded
by~\cite{pearson1990nearoptimal}:
\begin{align}
\label{eq:ubi}
\lsrnel\leq \left\lceil \sqrt{3(\lsrm-1)}~\right\rceil,~~~\lsrm-1 \geq 3.
\end{align}

Thus, in the non-periodic case ($\nblocks=1$), \thref{prop:1} reduces
our design problem to finding a length-$(\ulen-1)$ sparse ruler, for
which design algorithms abound.  A trivial example is
$\{0,\ldots,\ulen-1\}$, which clearly represents a universal sampler
since in that case $\by = \bx$. More sophisticated constructions were
discussed
in~\cite{redei1949representation,leech1956differences,pearson1990nearoptimal,wichmann1963difference,linebarger1993sparsearrays,pumphrey1993sparsearrays,pal2010nested}.
However, if the compression ratio is to be maximized, then one should
look for a \emph{minimal} sparse ruler, which is an exhaustive-search
problem. Fortunately, there exist tables for values of $\ulen-1$ up to
the order of 100. Although higher values of this parameter demand, in
principle, intensive computation, one may resort to the designs
in~\cite{pearson1990nearoptimal,wichmann1963difference,linebarger1993sparsearrays},
which provide nearly minimal rulers despite being really
simple.\removenow{\footnote{In any case, the design procedure chosen will not
  lead to any practical disadvantage in many cases since it does not
  affect the complexity of the sampling operation. }}

On the other hand, it is not clear how to design sampling patterns in
the periodic case ($\nblocks>1$) since periodicity needs to be enforced on
$\spat$. Before that, the next definition is required.

\begin{figure}[t]
 \centering
 \includegraphics[width=0.4\textwidth]{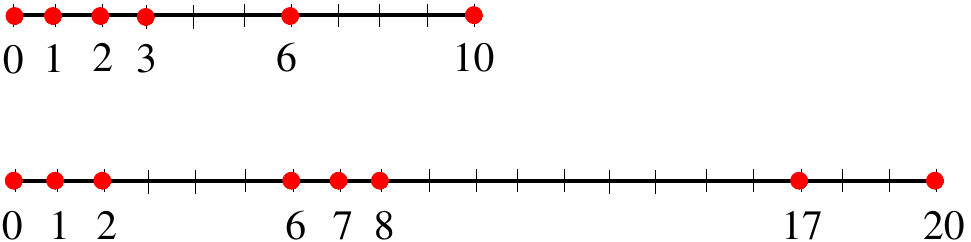}
 \caption{ Example of a length-10 minimal sparse ruler (above) and
   length-20 minimal sparse ruler (below).  }
 \label{fig:lsr}
\end{figure}

\begin{mydefinition}
\emph{ A length-$(\lsrm-1)$ periodic sparse ruler of period $\plen$, where
  $\plen$ divides $\lsrm$, is a set $\lsr\subset \{0,1,\ldots,\lsrm -1\}$
  satisfying two conditions:
\begin{enumerate}
\item if $\lsrel\in \lsr$, then $\lsrel+ \blockind \plen\in \lsr$ for all  $\blockind\in \zset$ such that
  $0\leq \lsrel+ \blockind \plen < \lsrm$
\item $\Delta(\lsr)=\{ 0,1,\ldots, \lsrm-1  \}$.
\end{enumerate}
It is called {minimal} if there exists no other periodic sparse ruler
with the same length and period but smaller cardinality.  }
\end{mydefinition}

Observe that any periodic sparse ruler is also a sparse ruler, whereas
the converse need not be true. Clearly, \thref{prop:1} could be
rephrased to say that $\spat$ is universal iff it is a
length-$(N\nblocks-1)$ periodic sparse ruler of period $N$. The
problem of designing {sparse} covariance samplers becomes that of
designing periodic sparse rulers. The next result simplifies this task
by stating that a length-$(N\nblocks-1)$ periodic sparse ruler of
period $N$ is indeed the concatenation of $\nblocks$ length-$(N-1)$
sparse rulers:

\begin{mytheorem}
\thlabel{prop:psr}
\emph{ A set $\spat$ is a periodic sparse ruler of length $N\nblocks-1$
  and period $N$ if and only if there exists a sparse ruler $\bspat$ of
  length $N-1$ such that }
\begin{align}
\label{eq:isetp}
\spat = \{ \bspatel + \blockind N:~~\bspatel\in
\bspat,~~\blockind=0,1,\ldots,\nblocks-1 \}.
\end{align}
\end{mytheorem}
\begin{proof}
See Appendix~\ref{sec:ppsr}.
\end{proof}

One of the consequences of \thref{prop:psr} is that increasing the
number of blocks in a periodic {sparse} sampler cannot improve the
compression ratio. For example, concatenating two equal length-$(N-1)$
minimal sparse rulers with $M$ elements results in a length-$(2N-1)$
sparse ruler with $2M$ elements. Note, however, that the situation is
different if the periodicity requirement is dropped. For instance, a
minimal length-$10$ sparse ruler has 6 elements, whereas a length-$21$
minimal sparse ruler has $8 < 6\times 2$ elements.

As a corollary of \thref{prop:psr}, we conclude that a minimal
periodic sparse ruler is the concatenation of minimal sparse
rulers. Thus, the problem of designing optimal {sparse} universal
covariance samplers (either periodic or non-periodic)  reduces
to designing a minimal length-$(N-1)$ sparse ruler $\bspat$.

 Table~\ref{tab:csr} illustrates the minimum value of $M = |\bspat|$
 (labeled as $M_\text{LSR}$) for several values of $N$, enabling us to
 obtain the optimum compression ratio for block lengths $N$ up to 60,
 which covers most practical cases.  For higher $N$, one may resort to
 another table, to a computer program, or to the asymptotic
 considerations in Sec.~\ref{sec:acr}. However, although there is no
 closed form expression for the maximum achievable compression ratio
 $\rho$, the bounds in \eqref{eq:lsrub} and \eqref{eq:ubi}~show~that
\begin{align}
\frac{N}{\left\lceil\sqrt{3(N-1)}~\right \rceil} \leq \rho \leq\frac{N}{\sqrt{\tau(N-1)}}.
\end{align}

\begin{table}[t]
\begin{center}
\begin{tabular}{|c |c |c |c |c |c |c |c |c |c |c |c |c |}
     \hline
 $N$ & 5 & 6 & 7 & 8 & 9 & 10 & 11 & 12 & 13 & 14 & 15 & 16 \\  \hline
 $M_\text{CSR}$ & 3 & 3 & 3 & 4 & 4 & 4 & 4 & 4 & 4 & 5 & 5 & 5 \\  \hline
 $M_\text{HLSR}$ & 3 & 3 & 3 & 4 & 4 & 4 & 4 & 4 & 4 & 5 & 5 & 5 \\  \hline
 $M_\text{LSR}$ & 4 & 4 & 4 & 5 & 5 & 5 & 6 & 6 & 6 & 6 & 7 & 7 \\  \hline
  \end{tabular}
\\[.1cm]
\begin{tabular}{|c |c |c |c |c |c |c |c |c |c |c |c |}
     \hline
 $N$ & 17 & 18 & 19 & 20 & 21 & 22 & 23 & 24 & 25 & 26 & 27 \\  \hline
 $M_\text{CSR}$ & 5 & 5 & 5 & 6 & 5 & 6 & 6 & 6 & 6 & 6 & 6 \\  \hline
 $M_\text{HLSR}$ & 5 & 5 & 5 & 6 & 6 & 6 & 6 & 6 & 6 & 6 & 6 \\  \hline
 $M_\text{LSR}$ & 7 & 7 & 8 & 8 & 8 & 8 & 8 & 8 & 9 & 9 & 9 \\  \hline
  \end{tabular}
\\[.1cm]
\begin{tabular}{|c |c |c |c |c |c |c |c |c |c |c |c |}
     \hline
 $N$ & 28 & 29 & 30 & 31 & 32 & 33 & 34 & 35 & 36 & 37 & 38 \\  \hline
 $M_\text{CSR}$ & 6 & 7 & 7 & 6 & 7 & 7 & 7 & 7 & 7 & 7 & 8 \\  \hline
 $M_\text{HLSR}$ & 7 & 7 & 7 & 7 & 7 & 7 & 7 & 7 & 8 & 8 & 8 \\  \hline
 $M_\text{LSR}$ & 9 & 9 & 9 & 10 & 10 & 10 & 10 & 10 & 10 & 10 & 11 \\  \hline
  \end{tabular}
\\[.1cm]
\begin{tabular}{|c |c |c |c |c |c |c |c |c |c |c |c |}
     \hline
 $N$ & 39 & 40 & 41 & 42 & 43 & 44 & 45 & 46 & 47 & 48 & 49 \\  \hline
 $M_\text{CSR}$ & 7 & 8 & 8 & 8 & 8 & 8 & 8 & 8 & 8 & 8 & 8 \\  \hline
 $M_\text{HLSR}$ & 8 & 8 & 8 & 8 & 8 & 8 & 8 & 8 & 8 & 9 & 9 \\  \hline
 $M_\text{LSR}$ & 11 & 11 & 11 & 11 & 11 & 11 & 12 & 12 & 12 & 12 & 12 \\  \hline
  \end{tabular}
\\[.1cm]
\begin{tabular}{|c |c |c |c |c |c |c |c |c |c |c |c |}
     \hline
 $N$ & 50 & 51 & 52 & 53 & 54 & 55 & 56 & 57 & 58 & 59 & 60 \\  \hline
 $M_\text{CSR}$ & 8 & 8 & 9 & 9 & 9 & 9 & 9 & 8 & 9 & 9 & 9 \\  \hline
 $M_\text{HLSR}$ & 9 & 9 & 9 & 9 & 9 & 9 & 9 & 9 & 9 & 9 & 10 \\  \hline
 $M_\text{LSR}$ & 12 & 12 & 13 & 13 & 13 & 13 & 13 & 13 & 13 & 13 & 14 \\  \hline
  \end{tabular}
  \end{center}
  \caption{Values of $M$ for a length-$(N-1)$  minimal circular sparse
    ruler ($M_\text{CSR}$), length-$\lfloor\frac{N}{2} \rfloor$
    minimal linear sparse ruler ($M_\text{HLSR}$) and length-$(N-1)$
    minimal linear sparse ruler ($M_\text{LSR}$).}
  \label{tab:csr}
\end{table}

\subsection{{Dense} Samplers}

Deriving conditions for universality of {dense} samplers is simpler
than for {sparse} samplers since most mathematical complexity
has been subsumed by \thref{prop:2}. Moreover, the results are simpler
and can be expressed in closed form.

\begin{mytheorem}
\thlabel{prop:rsu} \emph{ Let $\bphi$ be an $M\times N$ random matrix
  satisfying the hypotheses of \thref{prop:2}.  Then, $\bbphi=\bm
  I_{\nblocks}\otimes \bphi$ defines a universal covariance sampler with
  probability 1 if and only if}
\begin{align}
\label{eq:mbound}
M \geq \sqrt{
\frac{2N\nblocks-1}{2\nblocks-1}
}.
\end{align}
\end{mytheorem}
\begin{proof}
If $\basis$ is a basis for $\sspace$, then
$|\basis|=2\ulen-1=2N\nblocks-1$. From \thref{prop:2}, 
$\bbphi$ is an $\basis$-covariance sampler iff $2N\nblocks-1 \leq
M^2(2\nblocks-1)$, which is equivalent to
\eqref{eq:mbound}. Universality then follows from \thref{prop:bbasis}.
\end{proof}

 Expression \eqref{eq:mbound} can be interpreted as
 $M^2({2\nblocks-1}) \geq{2N\nblocks-1}$, where $2N\nblocks-1$ is the
 dimension of the uncompressed subspace and $M^2(2\nblocks-1)$ is the
 maximum dimension of a subspace of Hermitian block-Toeplitz
 matrices. Thus, this design provides optimal compression, which is
 achieved when
\begin{align}
M 
 = \left\lceil \sqrt{
\frac{2N\nblocks-1}{2\nblocks-1}
}\right\rceil, 
\end{align}
and given by
\begin{align}
\label{eq:rhonupb}
\rho = \frac{N}{M} \approx\sqrt{
\frac{(2\nblocks-1)N^2}{2N\nblocks-1}
}.
\end{align}

\section{Non-Universal Covariance Samplers}
\label{sec:nucs}

 Universal samplers are used when no structure exists or when it is
 unknown. However, when prior information is available, the values
 that $\bbsig$ can take on are restricted, allowing for larger
 compression ratios. This section analyzes this effect for the
 covariance subspaces introduced in Sec.~\ref{sec:rcs}. Since the
 Toeplitz subspace has already been considered in Sec.~\ref{sec:ucs},
 we proceed to analyze circulant and $\sndiag$-banded subspaces.

\subsection{Circulant Covariance Subspace}

\subsubsection{Sparse Samplers}
\label{sec:ccs:nus}
Restricting $\bsig$ to be circulant yields considerable compression
gains with respect to the Toeplitz case since the requirements on every
period of $\spat$ relax. In particular, every period must be a
\emph{circular} sparse ruler, which is a much weaker requirement than
that of being a \emph{linear} sparse ruler. This concept is related to
the \emph{modular} difference set defined next.  Recall from
Sec.~\ref{sec:not} that $( x )_\defmod$ denotes the remainder of the integer
division of $x$ by $\defmod$.
\begin{mydefinition}
\emph{ Let $\aset$ be a set of  integers.
  The $\defmod$-modular difference set of $\aset$, denoted as
  $\Delta_\defmod(\aset)$, is defined as
\begin{align}
\Delta_\defmod(\aset) = \{ \dif\geq 0: \exists a_1,a_2 \in \aset~ \st ~\dif=(a_2-a_1)_\defmod\}.
\end{align}
}
\end{mydefinition}
Clearly, for any $\aset \subset\{0,1,\ldots,\defmod-1\}$, we have that
$\Delta(\aset)\subset \Delta_\defmod(\aset)$, which means that
$|\Delta_\defmod(\aset)|$ is never less than $|\Delta(\aset)|$. Actually,
$\Delta_\defmod(\aset)$ will typically be larger than $\Delta(\aset)$ since
the fact that $\dif$ is in $\Delta_\defmod(\aset)$ implies that $\defmod-\dif$ is also in
that set.  For example, if $\aset=\{0,1,5\}$ and $\defmod=10$, then
$\Delta(\aset)=\{0,1,4,5\} \subset
\Delta_{10}(\aset)=\{0,1,4,5,6,9\}$.  Finally, the cardinality of
the modular difference set is upper bounded by noting that any pair
of elements in a set $\aset$ with cardinality $|\aset|$
generates at most two distances in $\Delta_\defmod(\aset)$:
\begin{align}
\label{eq:mdsb}
|\Delta_\defmod(\aset)| \leq  |\aset|\cdot (|\aset|-1) + 1.
\end{align}

Now 
it is possible
to state the requirements to compress circulant subspaces:
\begin{mytheorem}
\thlabel{prop:nuc} \emph{ Let $\basissc$ be given by
  \eqref{eq:bsetscdef}. Then, the set $\spat \subset\intset{0}{\ulen-1}$
  is an $\basissc$-covariance sampler if and only if $\Delta_\ulen(\spat)=\intset{0}{\ulen-1}$.}
\end{mytheorem}
\begin{IEEEproof}
See Appendix~\ref{sec:pnuc}.
\end{IEEEproof}

\thref{prop:nuc} is therefore the analogue of \thref{prop:1} for circulant
subspaces. However, in this case the conclusion does not lead  to a
\emph{linear} sparse ruler but to a \emph{circular} one:
\begin{mydefinition}
\label{def:csr}
\emph{ A length-$(\csrm-1)$ circular (or modular) {sparse ruler} is a set
  $\csr\subset \{0,\ldots,\csrm-1\}$ satisfying
  $\Delta_\csrm(\csr)=\{0,\ldots,\csrm-1\}$; and it is said to be {minimal}
  if no other length-$(\csrm-1)$ circular sparse ruler exists with smaller
  cardinality.}
\end{mydefinition}

As with linear sparse rulers, 
a geometric interpretation is possible in terms of a physical
ruler. Suppose that we wrap around a conventional ruler (made of some
flexible material) until the first mark and the last mark lie at unit
distance, thus making a \emph{circular ruler}. Now assume that some of
the marks are erased, but that it is still possible to measure all
distances between $0$ and the length of the original ruler using pairs
of marks. The advantage with respect to a linear ruler is that any
pair of marks provides, in general, \emph{two} distances, which are
the lengths of the two circular segments that they define.
Two length-$20$ circular sparse rulers are
illustrated in Fig.~\ref{fig:csr}, the one on the left being minimal.
Other examples of length-$(\csrm-1)$ circular sparse rulers are
$\{0,\ldots,\csrm-1\}$ and $\{0,\ldots,\lfloor\frac{\csrm}{2}
\rfloor\}$, which are referred to as \emph{trivial} circular sparse
rulers.

\begin{figure}[t]
 \centering

\begin{minipage}[b]{.23\textwidth}
 \centering
 \includegraphics[width=.9\textwidth]{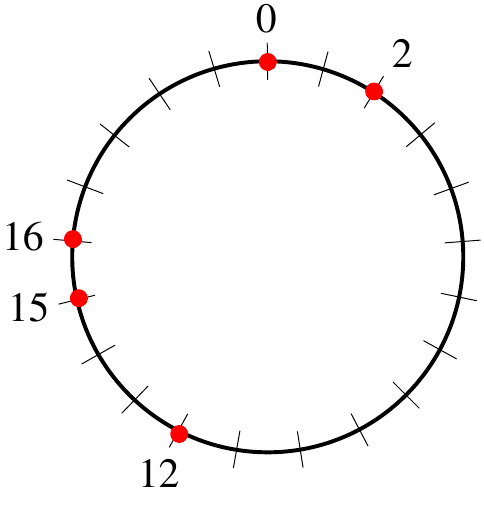}
 \subcaption{Example of length-20 minimal circular sparse ruler.\\~}
 \label{fig:csra}
\end{minipage}
~~
\begin{minipage}[b]{.23\textwidth}
 \centering
 \includegraphics[width=.9\textwidth]{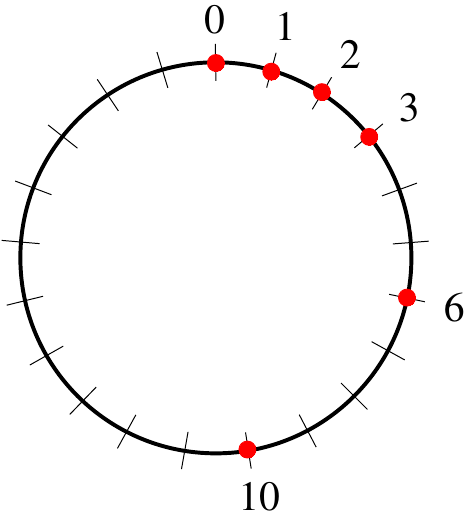}
 \subcaption{Example of length-20 circular sparse ruler designed with a length-10 \emph{linear} sparse ruler.}
 \label{fig:csrb}
\end{minipage}

  \caption{Comparison of two length-20 circular sparse rulers. The
    ruler on the left, with $5$ elements, is minimal whereas the
 one  on the right, with $6$ elements, is not. }
  \label{fig:csr}
\end{figure}


Circular sparse rulers, also known as \emph{difference cycles}, were
analyzed by the mathematical community using finite group theory and
additive number theory~(see \cite{miller1971three} for an overview of
the main results).  Among the most remarkable properties, we mention that
a reflection of a circular sparse ruler is also a circular sparse
ruler (see Sec.~\ref{sec:ucs:nus}) and that any circular rotation of a
circular sparse ruler $\csr$, defined as
\begin{align}
\csr_{(i)} =\{ (\csrel + i)_\csrm:~~\csrel\in \csr\},~~~i\in\zset,
\end{align}
is also a circular sparse ruler. Moreover, since $\Delta(\csr)\subset
\Delta_\csrm(\csr)$ for any $\csr\subset\intset{0}{\csrm-1}$, any
\emph{linear} sparse ruler is also a \emph{circular} sparse ruler.
Hence, the cardinality of a \emph{minimal} circular sparse ruler can
never be greater than the cardinality of a \emph{minimal} linear
sparse ruler if both have the same length. It is possible to go even
further by noting that any length-$\lfloor \frac{\csrm}{2}\rfloor$
linear sparse ruler is also a length-$(\csrm-1)$ circular sparse
ruler. For example, Fig. \ref{fig:csrb} shows a length-20 circular
sparse ruler constructed with a length-10 \emph{linear} sparse
ruler. From this observation and \eqref{eq:ubi}, we obtain
\begin{align}
\label{eq:ubcmcsr}
|\csr| \leq \left \lceil \sqrt{3 \left \lfloor \frac{\csrm}{2}\right\rfloor }~\right \rceil
\end{align}
for any minimal circular sparse ruler. On the other hand, expression
\eqref{eq:mdsb} yields
\begin{align}
\label{eq:lbcmcsr}
|\csr|\geq \frac{1}{2} + \sqrt{ \csrm - \frac{3}{4} }.
\end{align}

A length-$(\csrm-1)$ circular sparse ruler can be designed in several
ways.  For certain values of $\csrm$, minimal rulers attaining
\eqref{eq:lbcmcsr} can be obtained in closed form
(see~\cite[Sec.~III-B]{xia2005welch} for an overview; also
\cite{singer1938projective}). Other cases may require exhaustive
search, which motivates sub-optimal designs. Immediate choices are
length-$(\csrm-1)$ or length-$\lfloor \frac{\csrm}{2}\rfloor$ minimal
\emph{linear} sparse rulers~\cite{ariananda2012psd}. In fact, the
latter provides optimal solutions for most values of $\csrm$ below 60
(see Table~\ref{tab:csr}). 
Further alternatives include
\cite{dominguezjimenez2012design}.

Circular sparse rulers seem to have been introduced in signal/array
processing in~\cite{romero2013covariancesampling} and used later in
\cite{krieger2013multicosetarrays,ariananda2013periodogram,dominguezjimenez2013circular}.
\thref{prop:nuc} basically states that a covariance sampler for
circulant subspaces is a length-$(\ulen-1)$ circular sparse ruler,
which gives a practical design criterion just for the non-periodic
case. We now move on to introduce periodicity:
\begin{mydefinition}
\emph{ A length-$(\csrm-1)$ periodic circular sparse ruler of period
  $\plen$, where $\plen$ divides $\csrm$, is a set $\csr\subset
  \{0,1,\ldots,\csrm -1\}$ satisfying:
\begin{enumerate}
\item if $\csrel\in \csr$, then $\csrel+\blockind \plen\in \csr$ for all  $\blockind \in \zset$ such that
  $0\leq \csrel+\blockind \plen < \csrm$;
\item $\Delta_{\csrm}(\csr)=\{ 0,1,\ldots, \csrm-1  \}$.
\end{enumerate}
It is called {minimal} if there is no other periodic circular
sparse ruler with the same length and period but smaller cardinality.
}
\end{mydefinition}

Hence, \thref{prop:nuc} could be rephrased to say that $\spat$ is an
$\basissc$-covariance sampler iff it is a length-$(N\nblocks-1)$
periodic circular sparse ruler of period $N$. Although designing these
rulers may seem difficult, the next result simplifies this task by
stating that every period is, indeed, a circular sparse ruler.
\begin{mytheorem}
\thlabel{prop:cpsr}
\emph{ A set $\spat$ is a  periodic circular sparse ruler of length $N\nblocks-1$
  and period $N$ if and only if there exists a circular sparse ruler $\bspat$ of
  length $N-1$ such that }
\begin{align}
\spat = \{ \bspatel + \blockind N:~~\bspatel\in
\bspat,~~\blockind=0,1,\ldots,\nblocks-1 \}
\end{align}
\end{mytheorem}
\begin{proof}
See Appendix~\ref{sec:pcpsr}.
\end{proof}

\removenow{
\thref{prop:cpsr} is the dual of \thref{prop:psr} for circular sparse
rulers. Both theorems establish an interesting relationship, which can
be informally stated by saying that a periodic sparse ruler is a
concatenation of sparse rulers of the same nature, that is, a periodic
\emph{linear} sparse ruler is the concatenation of \emph{linear}
sparse rulers, whereas a periodic \emph{circular} sparse ruler is the
concatenation of \emph{circular} sparse rulers.
}

Table~\ref{tab:csr} reveals that the cardinality $M$ of a minimal
circular sparse ruler is not monotone with $N$. For example, minimal
length-19 circular sparse rulers have 6 elements whereas minimal
length-20 circular sparse rulers have 5 ele\-ments
(see~\cite{miller1971three} for a proof).  Table~\ref{tab:csr} also
illustrates the compression gain due to the knowledge that $\bsig$ is
circulant. For example, when $N=60$, a universal sampler has a
compression ratio of $\frac{N}{M} = \frac{60}{14}\approx 4.28$,
whereas a covariance sampler for circulant subspaces has a compression
ratio of $\frac{N}{M}=\frac{60}{9}\approx 6.67$.

Although  maximum compression ratios cannot be expressed in closed
form, simple bounds follow from 
\eqref{eq:ubcmcsr} and \eqref{eq:lbcmcsr}:
\begin{align}
\frac{N}{\left\lceil \sqrt{3\left\lfloor\frac{N}{2}\right\rfloor}~\right\rceil} \leq \rho \leq \frac{2N}{1+\sqrt{4N-3}}.
\end{align}

\subsubsection{{Dense} Samplers}
As in universal sampling, designing {dense} samplers is much easier
than designing {sparse} samplers. The following corollary of
\thref{prop:2} follows by noting that any basis for the circulant
subspace has $\ulen=N\nblocks$ elements.
\begin{mycorollary}
\emph{ Let $\bphi$ be an $M\times N$ random matrix satisfying the
  hypotheses of \thref{prop:2} and let $\basissc$ be given by
  \eqref{eq:bsetscdef}.  Then, with probability one, the matrix
  $\bbphi=\bm I_{\nblocks}\otimes \bphi$ defines an $\basissc$-covariance
  sampler if and only if }
\begin{align}
M \geq \sqrt\frac{N\nblocks}{2\nblocks-1}.
\end{align}
\end{mycorollary}
The optimum compression ratio is, therefore,
\begin{align}
\rho =\frac{N}{M} \approx \sqrt{\frac{(2\nblocks-1)N}{\nblocks}}.
\end{align}
For large $\nblocks$, this represents an approximate gain of $\sqrt 2$ with
respect to the universal case.

\subsection{$\sndiag$-banded Covariance Subspace}
\label{seq:bcm}

\subsubsection{Sparse samplers}

The prior knowledge that $\bsig$ is $\sndiag$-banded may also provide
compression gains.  In particular, we will see that, for {sparse}
samplers, $\sndiag$-banded subspaces with $N\leq \sndiag\leq
N(\nblocks-1)$ are compressed like circulant subspaces.
\begin{mytheorem}
\thlabel{prop:3} \emph{ Let $\basissb$ be given by \eqref{eq:bsetsbdef}
with $N\leq \sndiag\leq N(\nblocks-1)$. Then, the set 
\begin{align}
  \spat = \{ \bspatel+\blockind N,~\bspatel\in\bspat, \blockind=0,1,\ldots,\nblocks-1\},
\end{align}
where $\bspat \subset\intset{0}{N-1}$, defines an
$\basissb$-covariance sampler iff $\bspat$ is a length-$(N-1)$
circular sparse ruler.
}
\end{mytheorem}
\begin{IEEEproof}
See Appendix~\ref{sec:pprop:3}. 
\end{IEEEproof}

Observe that the condition $\sndiag\leq N(\nblocks-1)$ is a mild
assumption since we are only requiring that the last $N-1$ lags of the
associated autocorrelation sequence be zero.\footnote{Strictly
  speaking, we only need the lags $N\nblocks-N+1$ through
  $N\nblocks-1$ to be zero since the lags greater than $N\nblocks-1$
  are not relevant in the model.}  Note as well that other cases
rather than $N\leq \sndiag\leq N(\nblocks-1)$ may be considered,
resulting in different conclusions.  For example, in the non-periodic
case ($\nblocks=1$) it can be shown from \thref{prop:rmat} that the
only requirement on $\spat$ 
sampler is that $\Delta(\spat) = \intset{0}{\sndiag}$.

From \thref{prop:3} and \thref{prop:cpsr}, it follows that
$\spat$ must be a length-$(N\nblocks-1)$ periodic circular sparse
ruler of period $N$, which means that samplers for $\sndiag$-banded
subspaces mimic those for circulant subspaces. Thus, one should apply
the design and compression ratio considerations from
Sec.~\ref{sec:ccs:nus}. Interestingly, note that the latter does not
depend on $\sndiag$ provided that this parameter remains within the
aforementioned limits.

\subsubsection{{Dense} Samplers}
From \thref{prop:2} and noting that $\sndiag$-banded subspaces have
dimension $2\sndiag+1$ we obtain:
\begin{mycorollary}
\emph{ Let $\bphi$ be an $M\times N$ random matrix satisfying the
  hypotheses of \thref{prop:2} and let $\basissb$ be given by
  \eqref{eq:bsetsbdef}.  Then, with probability one, the matrix
  $\bbphi=\bm I_{\nblocks}\otimes \bphi$ defines an $\basissb$-covariance
  sampler if and only if }
\begin{align}
M \geq \sqrt\frac{2\sndiag+1}{2\nblocks-1}.
\end{align}
\end{mycorollary}
According to this result, the maximum compression ratio 
is: 
\begin{align}
\rho =\frac{N}{M} \approx \sqrt\frac{(2\nblocks-1)N^2}{2\sndiag+1},
\end{align}
which clearly improves the ratio in \eqref{eq:rhonupb} since
$\sndiag\leq N\nblocks-1$.

\section{Asymptotic Regime}
\label{sec:acr}

We next provide the optimal compression ratios for universal {dense}
samplers and bound the optimal ratios for universal {sparse} samplers
as $M$ and $N$ become larger.

\begin{itemize}
\item \textbf{{Dense} Samplers:} The maximum compression ratio
  $\rhors$ of universal {dense} samplers is given by
  \eqref{eq:rhonupb}.  Asymptotically in $N$, we have that $\rhors
  \rightarrow \sqrt{\frac{2\nblocks-1}{2\nblocks}N}$, which becomes
  $\rhors \rightarrow \sqrt{\frac{N}{2}}$ in the non-periodic case and
  $\rhors \rightarrow \sqrt{{N}}$ if the number of periods $\nblocks$
  also becomes large.  Alternatively, we observe that $M\rightarrow
  \sqrt{\frac{2\nblocks}{2\nblocks-1}N}$ as $N$ becomes large, which
  means that $M\rightarrow \sqrt{2N}$ in the non-periodic case and
  $M\rightarrow \sqrt{N}$ as $\nblocks\rightarrow \infty$.

\item \textbf{{Sparse} Samplers:} In
  \cite{leech1956differences,pearson1990nearoptimal} it is established
  that the quotient ${M^2}/{N}$ asymptotically converges to a constant
  $c$, which is between\footnote{As an informal guess,  consider
    the length-90 minimal sparse ruler, which has 16 elements. A
    simple approximation yields $c\approx 16^2/91 \approx 2.8132$.  }
  $\tau \approx 2.434$ and 3, with $M$ and $N-1$ respectively denoting
  the cardinality and length of a minimal linear sparse
  ruler. Therefore, the asymptotic optimal compression ratio is given by
\begin{align}
\label{eq:rhons}
\rhons \rightarrow  \sqrt{\frac{N}{c}}.
\end{align}
In terms of $M$, this means that $M\rightarrow
\sqrt{cN}$. Interestingly, if we use nested
arrays~\cite{pumphrey1993sparsearrays,pal2010nested}, the maximum
achievable compression we can obtain for suitable choices of the
parameters is $\rho_\text{NA} \rightarrow \sqrt\frac{N}{4}$, which is
therefore suboptimal. However, they present the advantage of having a
simple design. The scheme
in~\cite{wichmann1963difference,pearson1990nearoptimal} 
allows the simple construction of
sparse rulers satisfying ${M^2}/{N}<3$, which  entail compression
ratios greater than $\sqrt{\frac{N}{3}}$ even for finite $M$ and $N$.

\end{itemize}

To sum up, {dense} samplers provide better asymptotic compression
ratios than {sparse} samplers. The compression loss between both
approaches is quantified by the constant $c$, which means that between
36\% and 42\% compression may be lost for large $B$ if we use {sparse}
sampling instead of {dense} sampling.  Similar observations arise for
non-universal samplers by using the expressions in
Sec.~\ref{sec:nucs}.

Interestingly, these expressions show that the compression ratio can
be made arbitrarily large just by increasing the number of
observations. This conclusion agrees with \cite{masry1978poisson}.

\section{Discussion}
\label{sec:dis}

\change{The compression ratio was defined such that it is preserved
  for any number of realizations of $\by$ --- note that each one is compressed
  using that ratio.} In case of an arbitrarily large number of
realizations, the maximum compression ratio separates consistency from
inconsistency in the estimation. However, the notion of consistency is
not truly meaningful in case of just one realization. For those cases,
the values presented here provide simple guidelines to select suitable
compression ratios and a guess of the quality of the estimation, in
the sense that a good performance is expected when the actual
compression ratio is much lower than the maximum one and \emph{vice
  versa}.

\section{Conclusions}
\label{sec:con}

We have derived maximum compression ratios and optimal covariance
samplers for a number of cases including Toeplitz, circulant, and
banded covariance subspaces. The results were derived for the
general periodic case, but they can be immediately particularized to
the non-periodic setting. One of the effects observed is the
convenience of having long blocks.


Two common schemes were considered: {sparse} and {dense} samplers. The
design of optimal {sparse} samplers is related to the minimal sparse
ruler problem, which is an exhaustive search problem with known
near-optimal simple approximations.  Some cases deal
with \emph{linear} and others with \emph{circular} sparse rulers.

For {dense} samplers, the proposed random design is much simpler since
it solely depends on the size of the compression matrix relative to
the dimension of the covariance subspace. As opposed to the designs
presented for sparse samplers, which result in samplers which are
optimal \emph{only} among the family of sparse samplers, the random
designs proposed here result in samplers which are optimal in general,
that is, no other covariance sampler (either dense or sparse) can do
better. 

\appendices 

\section{Proof of  \thref{prop:aset}}
\label{sec:paset}
Clearly, if $\phi$ is invertible so is $\phir$. In order to prove
the converse statement, it suffices to show that $\phi$ is injective if
$\phir$ is injective. This is a simple consequence of the
definition of the codomains for both functions. Therefore, we need to
prove that, given any two vectors $\bm a =[a_0,\cdots, a_{\nbasis-1}]^T$ and $\bm
b= [b_0,\cdots, b_{\nbasis-1}]^T$ in $\rfield^\nbasis$, the matrices 
\begin{align}
\bsig_{\bm a} = \sum_\basisind a_\basisind \bsig_\basisind~~~\text{and}~~~
\bsig_{\bm b} = \sum_\basisind b_\basisind \bsig_\basisind
\end{align}
must satisfy that 
\begin{align}
\phi(\bsig_{\bm a}) = \phi(\bsig_{\bm b}) ~~\Rightarrow~~
\bsig_{\bm a} = \bsig_{\bm b}
\end{align}
or, equivalently, that
\begin{align}
\label{eq:phirtp}
\phi(\bsig_{\bm a}) = \phi(\bsig_{\bm b}) ~~\Rightarrow~~
\bm a = \bm b,
\end{align}
since $\basis$ is linearly independent. To do
so, let us take $\nbasis$ linearly independent vectors $\bm
\alpha_0,\cdots, \bm \alpha_{\nbasis-1}$, where $\bm \alpha_i =
      [\alpha_{i,0}\ldots\alpha_{i,\nbasis-1}]^T$, such that the $\nbasis$
      matrices
\begin{align}
\bsig_{\bm \alpha_i} = \sum_\basisind \alpha_{i,\basisind}\bsig_\basisind,~~~~i=0,\ldots,\nbasis-1
\end{align}
are in $\aset$. This operation is possible since
$\dimr[\aset \cap \spanvr \basis]=\nbasis$. Moreover, since $\phir$ is
injective and $\{\bsig_{\bm \alpha_i}\}_{i=0}^{\nbasis-1}$ is a linearly independent
set of matrices, it follows that the matrices
\begin{align}
\bbsig_{\bm \alpha_i}=\phir(\bsig_{\bm \alpha_i})=
\phi(\bsig_{\bm \alpha_i})
 = \sum_\basisind \alpha_{i,\basisind}\bbsig_\basisind
\end{align}
also form an independent set of matrices. On the other hand, since the
$\nbasis$ vectors $\bm \alpha_i$ constitute a basis for $\rfield^\nbasis$, it is
possible to write  $\bm a $ and $\bm b$ as:
\begin{align}
\bm a = \sum_i \tilde a_i \bm \alpha_i~~~~\text{and}~~~~
 \bm b = \sum_i \tilde b_i \bm \alpha_i, 
\end{align}
for some $\tilde a_i,\tilde b_i\in \rfield$, which in turn means that
\begin{align}
\bsig_{\bm a} = \sum_i \tilde a_i \bsig_{\bm
  \alpha_i}~~~~\text{and}~~~~
\bsig_{\bm b} = \sum_i \tilde b_i \bsig_{\bm \alpha_i}
\end{align}
or
\begin{align}
\phi(\bsig_{\bm a}) = \sum_i \tilde a_i \bbsig_{\bm
  \alpha_i}~~~~\text{and}~~~~
\phi(\bsig_{\bm b}) = \sum_i \tilde b_i \bbsig_{\bm \alpha_i}. 
\end{align}
Noting that the matrices $ \bbsig_{\bm \alpha_i}$ are linearly
independent leads to the statement
\begin{align}
\phi(\bsig_{\bm a}) = 
\phi(\bsig_{\bm b}) ~~~~\Rightarrow~~~~ \tilde a_i = \tilde b_i
~\forall i,
\end{align}
which is equivalent to \eqref{eq:phirtp}, thus concluding the proof.

\section{ Proof of \thref{prop:2}}
\label{sec:pp2}
In order to show \thref{prop:2} we will proceed by computing the
dimension of $\ker\phic$, and deriving the conditions under which
$\dim\ker\phic=0$, which, in virtue of \thref{prop:ker}, are the
conditions determining whether $\bbphi$ defines a covariance
sampler. However, since the direct computation of $\ker \phic$ is not
a simple task, we perform several intermediate steps. First, we
compute $\ker \tphis$, where $\tphis$ is defined as the extension of
$\phic$ to $\cfield^{\ulen\times \ulen}$:
   \begin{align}
\label{eq:tphisdef}
     \begin{array}{ccccc}
\cfield^{N\nblocks\times N\nblocks}   &\xrightarrow{~~\tphis~~} & \cfield^{M\nblocks\times M\nblocks} \\
  \bsig &\xrightarrow{~~~~~~} & \bbsig = \bbphi \bsig \bbphi^H
     \end{array}
   \end{align}
We later compute $\dim \ker \phic$ by successive intersections as
\begin{align}
\label{eq:kerphici}
\ker \phic = \spanvc \basis \cap \left(\tspace \cap \left(\bspace\cap \ker \tphis\right)\right),
\end{align}
where $\tspace$ represents the set of (not necessarily Hermitian) $N\nblocks
\times N\nblocks$ Toeplitz matrices and $\bspace$ represents the set of $N\nblocks
\times N\nblocks$  matrices with Toeplitz $N\times N$ blocks. The
matrices in $\bspace$ can thus be written as
\begin{align}
 \left[
\begin{array}{c c c c}
\bm A_{0,0}&,\cdots,&\bm A_{0,\nblocks-1}\\
\vdots&&\vdots\\
\bm A_{\nblocks-1,0}&,\cdots,&\bm A_{\nblocks-1,\nblocks-1}\\
\end{array}
\right]
\end{align}
where the blocks $\bm A_{\blockind,p}\in \cfield^{N\times N}$ are
Toeplitz. Expression \eqref{eq:kerphici} results from the fact
that $\ker \phic = \spanvc \basis \cap  \ker \tphis$ and
\begin{align}
\spanvc \basis \subset  \tspace \subset \bspace.
\end{align}

On the other hand, the requirement that the probability measure is
absolutely continuous
means that the probability that any row (or column) of $\bphi$ is in a
given subspace of dimension less than $N$ (resp. $M$) is zero. Another
consequence is that $\rank \bphi = M\leq N$ with probability one and,
as a result, the (right) null space of $\bphi$ has dimension
$N-M$. Let us denote by $\bV$ an $N\times( N-M)$ matrix whose columns
span this null space. Due to the properties of $\bphi$, it is clear
that the probability that the columns of $\bV$ are contained in a
given subspace of dimension less than $N$ is zero.

We start by computing a basis for  $\ker \tphis$ in terms of $\bV$.

\begin{mylemma}
\thlabel{prop:null} \emph{ Let $\eij\in \cfield^{\nblocks\times \nblocks}$ be the
  matrix with all entries set to zero but the $(i,j)$-th entry, which
  is one, and let $\bm e_k$ denote the $k$-th column of the identity
  matrix $\bm I_N$.  Let also $\tphis$ be defined as in
  \eqref{eq:tphisdef}, and let the columns of $\bm V =[\bv_0,\cdots,
    \bv_{N-M-1}]\in \cfield^{N\times (N-M)}$ form a basis for the null
  space of $\bphi$. Then, a basis for $\ker \tphis$ is given by}
\begin{align}
\wset=& \bigcup_{i=0}^{\nblocks-1}\bigcup_{j=0}^{\nblocks-1}\wset_{i,j},
\end{align}
\emph{where}
\begin{align}
\wset_{i,j}=& \Big\{
\eij\otimes \bm e_k \otimes \bm v_l^H,\\\nonumber
&~~~~ k=0,1,\ldots,N-1,~l=0,1,\ldots, N-M-1 \Big\}\\
&\cup
\Big\{
\eij\otimes \bm e_k^H \otimes \bm v_l,\\\nonumber
&~~~~ k=0,1,\ldots,M-1,~l=0,1,\ldots, N-M-1 \Big\}.
\end{align}
\end{mylemma}
\begin{IEEEproof}
See Appendix~\ref{sec:pprop:null}.
\end{IEEEproof}

Now let us evaluate the intersection $\bspace\cap \ker \tphis$, which
means that we must look for the matrices in $\ker \tphis$ whose
$N\times N$ blocks have a Toeplitz structure.  For the sake of
simplicity, let us proceed block-by-block by separately considering
the subspaces generated by each $\wset_{i,j}$. Clearly, the matrices
in $\spanvc\wset_{i,j}$ can have, at most, a single non-null
$N\times N$ block, which is the $(i,j)$-th block. This block is in the
subspace generated by the following basis:
\begin{align}
&  \{\nonumber
  \bm e_k \otimes \bm v_l^H,~
  k=0,1,\ldots,N-1,\\&~~~~~~~~~~~~~~~~~~~l=0,1,\ldots, N-M-1  \}
\nonumber\\\nonumber
&\cup
 \{
 \bm e_k^H \otimes \bm v_l,~ k=0,1,\ldots,M-1,\\&~~~~~~~~~~~~~~~~~~~~l=0,1,\ldots, N-M-1  \}.\nonumber
\end{align}
Therefore, all blocks in this subspace can be written in terms of this
basis as
\begin{align}
\sum_k \sum_l \alpha_{k,l}  (   \bm e_k \otimes \bm v_l^H )+
\sum_k \sum_l \beta_{k,l} (\bm e_k^H \otimes \bm v_l)
\end{align}
for some $\alpha_{k,l}\in \cfield$ and $\beta_{k,l}\in \cfield$. The
blocks with Toeplitz structure must necessarily satisfy
\begin{align}
\label{eq:condbt}
&\sum_{n=-N+1}^{N-1} \gamma_n \bm P_n = \sum_{k=0}^{N-1} \sum_{l=0}^{N-M-1} \alpha_{k,l}  (   \bm e_k \otimes \bm v_l^H ) \\&
~~~~~~~~~~~~~~~~~~+
\sum_{k=0}^{M-1} \sum_{l=0}^{N-M-1} \beta_{k,l} (\bm e_k^H \otimes \bm v_l)
\nonumber
\end{align}
for some  $\gamma_n\in \cfield$, where $\bm P_n$ equals
$\sidblen^n$ for  $n\geq 0$ and $(\sidblen^{-n})^T$ for
$n<0$, with $\sidblen$  defined in Sec.~\ref{sec:rcs}. 

Expression \eqref{eq:condbt} represents a system of linear equations
in $\alpha_{k,l}$, $\beta_{k,l}$ and $\gamma_{n}$, with $N^2
-M^2+2N-1$ unknowns and $N^2$ equations. On the other hand, since
$\bphi$, and consequently $\bV$, follow a continuous
distribution, it follows that there are $\min(N^2, N^2 -M^2+2N-1)$
independent matrices in \eqref{eq:condbt}. Consequently, if $N^2\geq N^2
-M^2+2N-1$ the only solution is just the zero matrix, and $\bspace \cap
\ker \tphis=\{\bm 0\}$, which in turn means that $\ker \phic=\{\bm
0\}$. Therefore, a sufficient condition for
$\bbphi$ to define a covariance sampler (see \thref{prop:ker}) is
\begin{align}
\label{eq:m2c1}
M^2 \geq 2N-1.
\end{align}
Conversely, if $N^2<N^2 -M^2+2N-1$ the
subspace of solutions has dimension $N^2 -M^2+2N-1-N^2=2N-M^2-1$.
Therefore, the  blocks of the  matrices in $\bspace \cap \ker \tphis$ can be
written as a linear combination of $ 2N-M^2-1$ Toeplitz matrices $\bm
M_k$.  By considering all blocks, it follows that $\bspace \cap \ker
\tphis$ is generated by the following basis:
\begin{align}
\nonumber
& \Big\{
\eij\otimes \bm M_k,~~i,j=0,1,\ldots,\nblocks-1;\\&~~~~~~~~~~~~~~~~~~~~~~k=0,1,\ldots,2N-M^2-2 \Big\}.
\end{align}
Thus, any matrix in $\bspace \cap \ker \tphis$ can be written as
\begin{align}
\bsig = \sum_{i,j,k}\eta_{i,j}^k \eij\otimes \bm M_k.
\end{align}

Now we compute the dimension of $\tspace \cap \left(\bspace\cap \ker
\tphis\right)$. First note that $\dim (\bspace\cap \ker \tphis) =
\nblocks^2(2N-M^2-1)$. In order for $\bsig\in \bspace\cap \ker \tphis $ to be
Toeplitz, we require that $\eta_{i,j}^k$ only depend on the
difference $i-j$, which reduces the dimension of this space to
$2N-M^2-1$ times the number of block diagonals, i.e.,
$(2N-M^2-1)(2\nblocks-1)$. Moreover, since any two adjacent block diagonals
share $N-1$ diagonals, this imposes $(2\nblocks-2)(N-1)$ additional equations
and results in
\begin{align}
&\dim \left(\tspace \cap \left(\bspace\cap \ker 
\tphis\right)\right) \\&~~~~~~=
(2N-M^2-1)(2\nblocks-1)-(2\nblocks-2)(N-1).  \nonumber
\end{align}

At this point, note that $\tspace$ is the smallest subspace containing
of both $\tspace \cap \bspace\cap \ker \tphis$ and $\spanvr
\basis$. Since $\bphi$ was generated according to a continuous
distribution, then with probability one these two subspaces will not
overlap (except for the zero matrix) unless the sum of their
dimensions exceeds the dimension of the parent subspace, which is
$2N\nblocks-1$. Therefore, $\bbphi$ defines a covariance sampler if
and only if
\begin{align*}
(2N-M^2-1)(2\nblocks-1)-(2\nblocks-2)(N-1) + \nbasis  \leq 2N\nblocks-1
\end{align*}
or, equivalently
\begin{align}
\label{eq:m2c2}
\nbasis  \leq M^2(2\nblocks-1).
\end{align}

It remains only to show that one only needs to look at 
\eqref{eq:m2c2} in order to assess whether a matrix $\bbphi$ defines a
covariance sampler, the condition in \eqref{eq:m2c1} being completely
irrelevant. This follows from the fact that \eqref{eq:m2c1} implies
\eqref{eq:m2c2}. Indeed, if we multiply both sides of \eqref{eq:m2c1}
by $(2\nblocks-1)$, we obtain
\begin{align}
M^2(2\nblocks-1)&\geq (2N-1)(2\nblocks-1)\\
&= (2N\nblocks-1)+2(N-1)(\nblocks-1) \\
&\geq (2N\nblocks-1) \geq \nbasis 
\end{align}
where the second inequality follows from the fact that $(N-1)(\nblocks-1)\geq
0$ and the third one is a consequence of the linear independence of
$\basis$.  Therefore, \eqref{eq:m2c1} implies \eqref{eq:m2c2}, and
$\bbphi$ defines a covariance sampler if and only if \eqref{eq:m2c2}
holds.

\section{Proof of \thref{prop:null}}
\label{sec:pprop:null}

Computing $\ker\tphis$ amounts to finding a basis for the subspace of
matrices $\bsig$ in $\cfield^{N\nblocks\times N\nblocks}$ satisfying
$\bbphi \bsig \bbphi^H=\bm 0$.  Vectorizing this expression (see
e.g.~\cite{bernstein2009}) results in the condition $(\bbphi^\ast
\otimes \bbphi)~ \bm z = \bm 0$, where $\bm z = \vecv \bsig$. Thus,
$\ker\tphis$ is given (up to inverse vectorization) by the null space
of the $(M\nblocks)^2\times (N\nblocks)^2$ matrix $\bbphi^\ast \otimes
\bbphi$.

Since the columns of $\bV$ constitute a basis for the null space of
$\bphi$ and since $\bbphi=\bm I_{\nblocks} \otimes \bphi$, the columns of
$\bbv = \bm I_{\nblocks} \otimes \bV$ constitute a basis for the null-space
of $\bbphi$. It can be shown that $\ker\tphis$ is composed of matrices
of the form $\bsig = \bbv \bm A^H + \bm B \bbv^H$, where $\bm A$ and
$\bm B$ are arbitrary matrices of the appropriate dimensions. It
follows that the null space of $\bbphi^\ast \otimes \bbphi$ is spanned
by the columns of the matrix
\begin{align}
\label{eq:bbwdef}
\bbw = [\bm I_{N\nblocks}\otimes \bbv,~~\bbv^\ast \otimes \bm I_{N\nblocks}].
\end{align}

By the properties of the Kronecker product~\cite{bernstein2009}, the
fact that $\bphi$ has maximum rank implies that $\bbphi^\ast \otimes
\bbphi$ has maximum rank as well, so that its null space has dimension
$(N^2-M^2)\nblocks^2$.  However, since $\bbv$ is $N\nblocks\times (N-M)\nblocks$, it is
clear that $\bbw$ has $2(N-M)N\nblocks^2$ columns, which is greater than
$(N^2-M^2)\nblocks^2$. Thus, in order to obtain a basis for the null space of
$\bbphi^\ast \otimes \bbphi$ we should remove dependent columns from
$\bbw$. This procedure is carried out in the following lemma:
\begin{mylemma}
\emph{ Let $\bV \in \cfield^{N\times(N-M)}$, with $M\leq N$, be a
  matrix whose columns generate the null space of $\bphi\in
  \cfield^{M\times N}$, which follows a continuous distribution, and
  let $\bbv = \bm I_{\nblocks} \otimes \bV$. Then, the columns of $\bbw$,
  defined by \eqref{eq:bbwdef}, span the same subspace as the columns
  of $\dbbw$, which is defined as }
\begin{align}
\label{eq:dbbwdef}
\dbbw = [\bm I_{N\nblocks}\otimes \bbv,~~\bbv^\ast \otimes \bm I_{\nblocks}\otimes\idzbsamps],
\end{align}
\emph{where $\idzbsamps = [\bm I_M, ~\bm
  0_{M,N-M}]^T$.}
\end{mylemma}
\begin{IEEEproof}
The procedure we follow in this proof is to remove linearly dependent
columns from $\bbw$.  Since the case $\nblocks>1$ is quite tedious, here we
only show this result for the case $\nblocks=1$. The proof for the general
case follows the same lines and it is easily extrapolated, but it
requires an overloaded notation.  For $\nblocks=1$ we have that
\begin{align}
\bbw = [\bm I_{N}\otimes \bV,~~\bV^\ast \otimes \bm I_{N}].
\end{align}
Now scale the last $N(N-M)$  columns of $\bbw$ to
obtain
\begin{align}
\bbw' = [\bm I_{N}\otimes \bV,~~\bm G \otimes \bm I_{N}],
\end{align}
where $\bm G$ is the result of scaling the columns of $\bm V^\ast$
such that the first row contains only ones\footnote{This is always
  possible whenever the elements of the first row of $\bV$ are all
  different from zero. However, it is possible with probability one to
  choose  $\bV$ such that it generates the null space of
  $\bphi$ and satisfies this condition.}:
\begin{align}
\bm G =
\left[
\begin{array}{c c c c}
1& 1 & \ldots & 1\\
g_{1,0}& g_{1,1} & \ldots & g_{1,N-M-1}\\
\vdots & \vdots &\ddots & \vdots \\
g_{N-1,0}& g_{N-1,1} & \ldots & g_{N-1,N-M-1}\\
\end{array}
\right]
\end{align}
Now consider a submatrix of $\bbw'$ obtained by retaining the first
$N(N-M)$ columns and the columns with indices $N(N-M) +
Ni,\ldots,N(N-M) + N(i+1) -1$, i.e., 
\begin{align}
\bbw_i' = 
\left[
\begin{array}{c c c c c}
  \bV & \bm 0 & \ldots & \bm 0 & \bm I_{N}\\
\bm 0 & \bV &  \ldots & \bm 0& g_{1,i}\bm I_{N}\\
 &  & \ddots & \\
\bm 0 & \bm 0 &  \ldots & \bV & g_{N-1,i}\bm I_{N}\\
\end{array}
\right], 
\end{align}
where $i=0,\ldots,N-M-1$.  Scaling the diagonal blocks on the left yields:
\begin{align}
\label{eq:bbwi}
\bbw_i'' = 
\left[
\begin{array}{c c c c c}
  \bV & \bm 0 & \ldots & \bm 0 & \bm I_{N}\\
\bm 0 & g_{1,i}\bV &  \ldots & \bm 0&  g_{1,i}\bm I_{N}\\
 &  & \ddots & \\
\bm 0 & \bm 0 &  \ldots & g_{N-1,i}\bV & g_{N-1,i}\bm I_{N}\\
\end{array}
\right].
\end{align}
Now, since $\bphi$ follows a continuous distribution, the
last $N-M$ columns of  $[\bV,\bm I_N]$ can be written as
linear combinations of the first $N$ columns, which means that the
last $N-M$ columns of $\bbw_i'$ are linearly dependent of the
others. Repeating this operation for $i = 0,\ldots, N-M-1$ and
removing from $\bbw$ the columns declared as dependent at each $i$
gives
\begin{align}
\dbbw = [\bm I_{N}\otimes \bV,~~\bV^\ast \otimes \idzbsamps],
\end{align}
which clearly spans the same subspace as $\bbw$. In the general case with
$\nblocks\geq 1$ we obtain
\begin{align}
\dbbw = [\bm I_{N\nblocks}\otimes \bbv,~~\bbv^\ast \otimes \bm I_{\nblocks}\otimes\idzbsamps].
\end{align}
\end{IEEEproof}

Note that, indeed, the matrix defined in \eqref{eq:dbbwdef} has
$(N^2-M^2)\nblocks^2$ columns, which means that they constitute a basis for
the null space of $\bbphi^\ast \otimes \bbphi$. Upon inverse
vectorization of the columns of $\dbbw$ we obtain the sought basis in
matrix form:
\begin{align}
\wset=& \Big\{
\eij\otimes \bm e_k \otimes \bm v_l^H,~~ i,j=0,1,\ldots
\nblocks-1,\\&~~~~~~~~~k=0,1,\ldots,N-1,~l=0,1,\ldots, N-M-1 \Big\}\nonumber
\\
&\cup
\Big\{
\eij\otimes \bm e_k^H \otimes \bm v_l,
~~ i,j=0,1,\ldots
\nblocks-1,\\&~~~~~~~~~k=0,1,\ldots,M-1,~l=0,1,\ldots, N-M-1 \Big\}.\nonumber
\end{align}

\section{Proof of \thref{prop:psr}}
\label{sec:ppsr}

Clearly, if $\bspat$ is a length-$(N-1)$ sparse ruler, then
\eqref{eq:isetp} defines a periodic sparse ruler. To show the converse
statement, assume that $\spat$ is a periodic sparse ruler and take
$\bspat = \spat \cap \{0,\ldots, N-1\}$. Then,
$\{0,\ldots,N\nblocks-1\}\subset \Delta(\spat)$ and, in particular,
$\{N(\nblocks-1),\ldots,N\nblocks-1\}\subset \Delta(\spat)$, meaning that
\begin{align}
&\forall \dif  \in \{N(\nblocks-1),\ldots,N\nblocks-1\},~\exists
  \aelt,\aelo\in \spat~\text{s.t.}~\aelt-\aelo = \dif.
\nonumber
\end{align}
Due to the periodicity of $\spat$, any $\spatel\in \spat$ can be
uniquely decomposed as $ \spatel = \bspatel_\spatel + \blockind_\spatel N$, where $\bspatel_\spatel\in \bspat$ and
$\blockind_\spatel \in \{0,\ldots,\nblocks-1\}$. Denote as $\bspatel_\aelo,~\bspatel_\aelt,~\blockind_\aelo$ and $\blockind_\aelt$ the
corresponding coefficients of the decomposition of $\aelo$ and $\aelt$.
Therefore, the condition above becomes
\begin{align}
&\forall \dif \in \{N(\nblocks-1),\ldots,N\nblocks-1\},~
\exists \bspatel_\aelo,\bspatel_\aelt\in \bspat\\&~~\text{and}~~\blockind_\aelo,\blockind_\aelt\in   \{0,\ldots,\nblocks-1\}~\st~\bspatel_\aelt-\bspatel_\aelo+(\blockind_\aelt-\blockind_\aelo)N = \dif.\nonumber
\end{align}
Since $\bspatel_\aelt-\bspatel_\aelo\leq N-1$ and $\dif\geq N(\nblocks-1)$, it is clear that $\blockind_\aelt-\blockind_\aelo$
must equal $\nblocks-1$  in order for the condition $\bspatel_\aelt-\bspatel_\aelo+(\blockind_\aelt-\blockind_\aelo)N
= \dif$ to hold. Then, after subtracting $N(\nblocks-1)$, the following
equivalent expression arises:
\begin{align}
&\forall \dif \in \{0,\ldots,N-1\},~
\exists \bspatel_\aelo,\bspatel_\aelt\in
\bspat~\st~\bspatel_\aelt-\bspatel_\aelo = \dif.
\nonumber
\end{align}
Hence, $\bspat$ is a sparse ruler.

\section{Proof of \thref{prop:nuc}}
\label{sec:pnuc}

Assume that $\ulen$ is odd. The proof for $\ulen$ even follows similar
lines.  If $\Delta(\spat) = \intset{0}{\ulen-1}$, then the matrix from
\thref{prop:rmat} is given by:
\begin{align}
\bm R = 
\left[
\begin{array}{c c c}
1 & \zvl^T & \zvl^T\\
\zvl & \bm I_{\sulen}  & -\jmath\bm I_{\sulen}\\
\zvl & \bm K_{\sulen}  & \jmath\bm K_{\sulen}\\
1 & \zvl^T & \zvl^T\\
\zvl & \bm I_{\sulen}  & \jmath\bm I_{\sulen}\\
\zvl & \bm K_{\sulen}  & -\jmath\bm K_{\sulen}
\end{array}
\right]
\end{align}
where $\sulen = \frac{\ulen-1}{2}$, $\zvl$ is an $\sulen\times 1$
vector with all zeros and $\bm K_\sulen$ is an $\sulen\times\sulen$
Hankel matrix with ones in the antidiagonal and zeros elsewhere, i.e.,
its $(m,n)$-th element equals 1 if $m+n=\sulen-1$ and 0 otherwise. All
the columns are linearly independent so that $\rank\bm R=\ulen$ and,
according to \thref{prop:rmat}, $\spat$ is an $\basissc$-covariance
sampler.

Now consider removing elements from $\Delta(\spat)$.  It can readily be 
seen that the rank is not maximum iff there is some
$\delta\in\intset{0}{\ulen}$ such that $\delta\notin\Delta(\spat)$ and 
 $\ulen-\delta\notin\Delta(\spat)$. Equivalently, we can say that the
rank is maximum if and only if $\Delta_\ulen(\spat)=\intset{0}{\ulen-1}$.

\section{Proof of \thref{prop:cpsr}}
\label{sec:pcpsr}

Let us start by showing that if $\bspat$ is a circular sparse ruler,
then $\spat$ is a periodic circular sparse ruler or, in other words,
if $\Delta_N(\bspat)=\intset{0}{N-1}$, then
$\Delta_{N\nblocks}(\spat) = \intset{0}{N\nblocks-1}$. Consider any
$\dif \in\intset{0}{N-1}$. Since $\dif \in\Delta_N(\bspat)$, 
at least one of the following two conditions will hold:
\begin{align}
&\text{C1:~~}\exists \bspatel_1,\bspatel_2\in\bspat,~\bspatel_2\geq \bspatel_1 ~\text{such that}\\\nonumber
&~~~~~~~~~~~~~~~~~~~~~~~~~~~~~~~ ~(
  \bspatel_2-\bspatel_1)_N = \bspatel_2-\bspatel_1 =
  \dif \\
&\text{C2:~~}\exists \bspatel_1,\bspatel_2\in\bspat,~\bspatel_2<\bspatel_1 ~\text{such that}\\
&~~~~~~~~~~~~~~~~~~~~~~~~~~~~~~~ ~(
  \bspatel_2-\bspatel_1)_N = N+ \bspatel_2-\bspatel_1 =
  \dif \nonumber
\end{align}
We next show that, in both cases, all the elements of the form $\dif  +
\blockind N$, with $\blockind=0,\ldots,\nblocks-1$, are in $\Delta_{N\nblocks}(\spat)$:
\begin{itemize}
\item C1: consider $\spatel_2 = \bspatel_2+\blockind N$ and $\spatel_1=\bspatel_1$ for any
  $\blockind=0,\ldots,\nblocks-1$. Since $\spatel_1,\spatel_2\in \spat$, it follows that
  $(\spatel_2-\spatel_1)_{N\nblocks} = \bspatel_2+\blockind N -\bspatel_1 = \dif +\blockind N\in \Delta_{N\nblocks}(\spat)$.
 \item C2: first take $\spatel_1 = \bspatel_1$ and $\spatel_2 = \bspatel_2+N+\blockind N$ with
   $\blockind=0,\ldots,\nblocks-2$.  Since $\spatel_1,\spatel_2\in \spat$, then $(\spatel_2-\spatel_1 )_{N\nblocks}=
   \bspatel_2+N+\blockind N-\bspatel_1 = \dif  + \blockind N\in \Delta_{N\nblocks}(\spat)$. It suffices only to
   show that $\dif  + \blockind N\in \Delta_{N\nblocks}(\spat)$ when $\blockind=\nblocks-1$. To this end,
   consider $\spatel_1=\bspatel_1$ and $\spatel_2=\bspatel_2$, which results in $(\spatel_2-\spatel_1)_{N\nblocks}
   = N\nblocks+\bspatel_2-\bspatel_1 = N(\nblocks-1) + N + \bspatel_2-\bspatel_1 = N(\nblocks-1) + \dif  \in
   \Delta_{N\nblocks}(\spat)$. 
\end{itemize}
To sum up, we have shown that $\dif +\blockind N \in    \Delta_{N\nblocks}(\spat)$ for any
$\dif =0,\ldots,N-1$ and $\blockind = 0,\ldots,\nblocks-1$, which establishes that
$\spat$ is a circular sparse ruler.

To show the converse statement, assume that $\spat$ is a circular
sparse ruler, i.e., $ \Delta_{N\nblocks}(\spat) =
\intset{0}{N\nblocks-1}$. In particular, all modular distances of the
form $\dif = \intset{0}{N-1}$ are present in $
\Delta_{N\nblocks}(\spat) $, which means that one or both of the
following two conditions will be satisfied:
\begin{align}
&\text{C1':~~}\exists \spatel_1,\spatel_2\in\spat,~\spatel_2\geq \spatel_1 ~\text{such that}\\
&~~~~~~~~~~~~~~~~~~~~~~~~~~~~~~~ ~(
  \spatel_2-\spatel_1)_{N\nblocks} = \spatel_2-\spatel_1 =  \dif,
  \nonumber \\ 
&\text{C2':~~}\exists \spatel_1,\spatel_2\in\spat,~\spatel_2<\spatel_1 ~\text{such that}\\
&~~~~~~~~~~~~~~~~~~~~~~~~~~~~~~~ ~(
  \spatel_2-\spatel_1)_{N\nblocks} = N\nblocks+\spatel_2-\spatel_1 =   \dif \nonumber.
\end{align}
It is therefore to be shown that $\dif \in \Delta_N(\bspat)$ in both cases,
where $\bspat$ is defined as $\bspat = \spat \cap \intset{0}{N-1}$.
\begin{itemize}
\item C1': clearly, we can assume without any loss of
  generality that $\spatel_1\in \bspat$. According to whether $\spatel_2$ is also
  in $\bspat$ or not, we distinguish two scenarios:
\begin{itemize}
\item $\spatel_2\in \bspat$: in this case, it is clear that $(\spatel_2-\spatel_1)_N = \spatel_2-\spatel_1 \in
  \Delta_N(\bspat)$. 
\item $\spatel_2\notin \bspat$: since $0\leq \dif < N$, it follows
  that $\spatel_2$ can be written as $\spatel_2 = \bspatel + N$ for
  some $\bspatel\in \bspat$ with $\bspatel<\spatel_1$. Therefore,
  $(\bspatel-\spatel_1)_N = N+\bspatel-\spatel_1 = \spatel_2-\spatel_1
  = \dif \in \Delta_N(\bspat)$.
\end{itemize}
 \item C2': since $0\leq \dif < N$, it can be seen that
   $N(\nblocks-1)<\spatel_1-\spatel_2\leq N\nblocks$, which in turn
   requires $\spatel_2\in \bspat$ and $\spatel_1 = \bspatel +
   N(\nblocks-1)$ for some $\bspatel\in \bspat$ with
   $\bspatel>\spatel_2$. Now consider the circular distance between
   $\bspatel$ and $\spatel_2$:
\begin{align*}
   (\spatel_2-\bspatel)_N &= N+\spatel_2-\bspatel = N + \spatel_2 - [\spatel_1-N(\nblocks-1)]\\
& = \spatel_2 - \spatel_1 + N\nblocks = \dif  \in \Delta_N(\bspat).
\end{align*}
\end{itemize}
Therefore, we have shown that $\dif \in \Delta_N(\bspat)$ for all
$\dif =0,\ldots,N-1$, which means that $\bspat$ is a circular sparse ruler.

\section{Proof of \thref{prop:3}}
\label{sec:pprop:3}
If we form the matrix $\bm R$ in \thref{prop:rmat}  using the matrices from
\eqref{eq:bsetsbdef}, we conclude that $\{0,\ldots,\sndiag\}\subset
\Delta(\spat)$ in order for $\spat$ to define a covariance sampler. As
we did to prove \thref{prop:psr}, we write the following necessary and sufficient
condition:
\begin{align}
&\forall \dif \in \{0,\ldots, \sndiag\},~
\exists \bspatel_\asindo,\bspatel_\asindt\in \bspat~\text{and}~~\blockind_\asindo,\blockind_\asindt\in\nonumber
\{0,\ldots,\nblocks-1\}\\&~~~~~~~\text{such that}~\bspatel_\asindt-\bspatel_\asindo+(\blockind_\asindt-\blockind_\asindo)N = \dif.\label{eq:condx}
\end{align}

We start by showing that if $\bspat$ is a circular sparse ruler, then
\eqref{eq:condx} holds true, i.e., $\spat$ is a covariance
sampler. More specifically, we show that $\dif\in \Delta(\spat)~\forall
\dif\in \{0,\ldots,N(\nblocks-1)\}$.  Consider two different cases:
\begin{itemize}
\item Case $0\leq \dif <N(\nblocks-1)$: It suffices to write $\dif$ as
  $\dif = \bspatel_\dif + \blockind_\dif N$, with $\bspatel_\dif\in
  \{0,\ldots,N-1\}$ and
  $\blockind_\dif\in\{0,\ldots,\nblocks-2\}$. Since $\bspatel_\dif\in
  \Delta_N(\bspat)$, then $\bspatel_\dif$ can be represented either as
  $\bspatel_{\dif,\asindt}-\bspatel_{\dif,\asindo}$ or as
  $N+\bspatel_{\dif,\asindt}-\bspatel_{\dif,\asindo}$, where
  $\bspatel_{\dif,\asindo},\bspatel_{\dif,\asindt}\in\bspat$. In the
  former case just make $\bspatel_\asindt = \bspatel_{\dif,\asindt}$,
  $\bspatel_\asindo = \bspatel_{\dif,\asindo}$, $\blockind_\asindt =
  \blockind_\dif $ and $\blockind_\asindo = 0$. In the later case make
  $\bspatel_\asindt = \bspatel_{\dif,\asindt}$, $\bspatel_\asindo =
  \bspatel_{\dif,\asindo}$, $\blockind_\asindt = \blockind_\dif +1$
  and $\blockind_\asindo = 0$.
\item Case $\dif  = N(\nblocks-1)$: this is trivial since $N(\nblocks-1)\in
  \Delta(\spat)$ for any non-empty choice of $\bspat$. 
\end{itemize}

Now, in order to show the converse theorem, we prove that if
$\intset{0}{N-1}\subset \Delta(\spat)$, then $\intset{0}{N-1}\subset
\Delta_N(\bspat)$. Let us consider some $\dif \in
\intset{0}{N-1}$. Since $\dif \in \Delta(\spat)$, it is clear that
there exist some $\bspatel_\asindo,\bspatel_\asindt\in \bspat$ and
$\blockind_\asindo,\blockind_\asindt\in\{0,\ldots,\nblocks-1\}$ such that
$\bspatel_\asindt-\bspatel_\asindo+(\blockind_\asindt-\blockind_\asindo)N = \dif
$. In particular, $(\blockind_\asindt-\blockind_\asindo)$ can be either 0 or
1. Therefore, for any $\dif \in \intset{0}{N-1}$, there exists
$\bspatel_\asindo,\bspatel_\asindt\in \bspat$ such that either $ \bspatel_\asindt
-\bspatel_\asindo=\dif $ or $N+\bspatel_\asindt-\bspatel_\asindo=\dif
$. Noting that this condition is equivalent to the condition
$\intset{0}{N-1}\subset \Delta_N(\bspat)$ concludes the proof.

\bibliographystyle{IEEEbib}
\bibliography{my_bibliography}

\end{document}